\documentclass[12pt,a4paper]{article}
\usepackage[titletoc,title]{appendix}
\usepackage{amsmath,amsfonts,amssymb,bm,amsthm}
\usepackage[colorlinks,citecolor=blue,linkcolor=blue]{hyperref}
\usepackage{tikz}
\usetikzlibrary{shapes,arrows}
\usetikzlibrary{intersections,shapes.arrows}

\usepackage[T1]{fontenc}
\usepackage[top=1in, bottom=1in, left=1in, right=1in]{geometry}
\usepackage{fancyhdr}
\newcommand{\dd}{\mathrm{d}}

\usepackage{enumerate}
\usepackage{xcolor,cancel}

\newcommand{\dist}{\operatorname{dist}}

\definecolor{NiColor}{RGB}{77,77,255}
\definecolor{NiColoRed}{RGB}{255,77,77}
\definecolor{NiCitation}{RGB}{0,181,26}

\numberwithin{equation}{section}
\newtheorem{theorem}{Theorem}[section]
\newtheorem{lemma}[theorem]{Lemma}
\newtheorem{proposition}[theorem]{Proposition}
\newtheorem{corollary}[theorem]{Corollary}
\theoremstyle{definition}

\newtheorem{remark}[theorem]{Remark}
\newtheorem{definition}[theorem]{Definition}
\newtheorem{notation}[theorem]{Notation}

\begin{document}

\title{Global wave parametrices \\
on globally hyperbolic spacetimes}
\author{
	Matteo Capoferri\thanks{MC:
Department of Mathematics,
University College London,
Gower Street,
London WC1E~6BT,
UK;
\mbox{matteo.capoferri@gmail.com}.
\emph{Current address:} School of Mathematics, Cardiff University, Senghennydd Road, Cardiff CF24 4AG, UK.
}
\and
Claudio Dappiaggi\thanks{CD:
Dipartimento di Fisica,
Universit\`a degli Studi di Pavia \& INFN, Sezione di Pavia, 
Via Bassi 6,
I-27100 Pavia,
Italia;
\mbox{claudio.dappiaggi@unipv.it}
}
\and
Nicol\`o Drago \thanks{ND: Department of Mathematics, 
	Julius Maximilian University of W\"urzburg,
	Emil-Fischer-Stra\ss e 31,
	D-97074 W\"urzburg,
	Germany;
\mbox{nicolo.drago@mathematik.uni-wuerzburg.de}
}}

\renewcommand\footnotemark{}

\date{29 June 2020}

\maketitle
\begin{abstract}
	
In a recent work the first named author, Levitin and Vassiliev have constructed the wave propagator on a closed Riemannian manifold $M$ as a single oscillatory integral global both in space and in time with a distinguished complex-valued phase function. In this paper, first we give a natural reinterpretation of the underlying algorithmic construction in the language of ultrastatic Lorentzian manifolds. Subsequently we show that the construction carries over to the case of static backgrounds thanks to a suitable reduction to the ultrastatic scenario. Finally we prove that the overall procedure can be generalised to any globally hyperbolic spacetime with compact Cauchy surfaces. As an application, we discuss how, from our procedure, one can recover the local Hadamard expansion which plays a key role in all applications in quantum field theory on curved backgrounds.\\


\

{\bf MSC classes: } primary 58J40; secondary 35L05, 53C50, 35Q40.

\end{abstract}

\tableofcontents

\section{Main results}
\label{Main results}

The study of hyperbolic partial differential equations on manifolds, particularly with reference to the D'Alembert wave operator, has been and continues to be, in various forms, at the forefront of scientific research. This is certainly true for mathematics, where the study of wave propagation has led over the years to major breakthroughs in the pure and in the applied side of the subject alike. 

In this context, of particular relevance is the so-called wave propagator which is defined as follows. Consider $(\Sigma,h)$ a connected closed smooth Riemannian manifold of dimension $\dim\Sigma\geq 2$ and let $\Delta$ be the Laplace--Beltrami operator associated with $h$. The latter is known to be a non-positive operator, thus the square root $\sqrt{-\Delta}$ is well-defined.  We call {\em wave propagator} the Fourier Integral Operator (FIO) $U(t)$, $t\in\mathbb{R}$, the distributional solution of the operator (pseudodifferential) initial value problem
\begin{equation}\label{eq: starting_point}
\left(-i\frac{\partial}{\partial t}+\sqrt{-\Delta}\right)U(t)=0,\quad U(0)=\operatorname{Id}.
\end{equation}
It is well-known that the knowledge of $U(t)$ is sufficient to solve the Cauchy problem associated with the D'Alembert wave operator $\frac{\partial^2}{\partial t^2}-\Delta$. For this reason the study of the wave propagator has attracted a lot of attention, especially by means of microlocal-analytic techniques, see, {\it e.g.}, \cite{Hor, shubin}.

In this paper we are interested instead in a different approach, which aims at a global construction of the wave propagator. This is based on the works of Laptev, Safarov Vassiliev \cite{LSV} and of Safarov and Vassiliev \cite{SaVa} in which it has been shown that the propagator for a wide class of hyperbolic equations can be written as a single FIO, global both in space and in time, with complex-valued -- as opposed to real-valued -- phase function. This approach has the advantage of circumventing obstructions due to caustics. Such viewpoint has recently been adopted in \cite{CLV}, where the authors proposed a global invariant definition of the full symbol of the wave propagator, together with an algorithm for the explicit calculation of its homogeneous components. It is worth observing that the compactness assumption on $\Sigma$ could be relaxed with quite some effort. In this paper we build upon \cite{CLV}, requiring that all underlying Riemannian manifolds are closed, in order to avoid unnecessary technical difficulties, focusing instead on the main ideas and constructions.

Our investigation originates from two key observations. On the one hand, the study of hyperbolic partial differential equations sees one of its main applications in the description of physical phenomena modelled over Lorentzian manifolds. These are the natural playground of several theories, such as general relativity and quantum field theory on curved backgrounds, whose mathematical properties have been for decades the subject of thorough investigations. On the other hand, the reinterpretation of the framework depicted in \cite{CLV} in the language of Lorentzian geometry is rather natural, particularly with reference to a specific class of backgrounds known as {\em ultrastatic spacetimes}.

In this paper we will move from the observations above, asking ourselves to which extent it is possible to adapt and generalise the results of \cite{CLV} to the Lorentzian setting. In doing so, we will focus on the construction of the wave propagator on a distinguished class of spacetimes, known as {\em globally hyperbolic}, see, {\it e.g.}, \cite{Bar:2007zz}. These are of notable importance since they represent the natural collection of Lorentzian manifolds on which the Cauchy problem for hyperbolic partial differential equations of the like of the D'Alembert wave equation is well-posed. It is important to stress that, by far, this is not the first paper in which FIOs are used in a Lorentzian framework to analyse distinguished parametrices, see, {\it e.g.}, \cite{Junker:2001gx,Strohmaier:2018bcw}. Yet, our method differs significantly from those used in earlier publications, especially with regards to the invariant approach relying on the use of a distinguished complex-valued phase function.

The main goal of this paper is to construct explicitly, in terms of Fourier integral operators, solutions to the wave equation on a Lorentzian spacetime $(M,g)$ with compact Cauchy surface. More precisely, we will construct, modulo an operator with infinitely smooth Schwartz kernel, the solution operator mapping initial data on a given Cauchy surface to full solutions in $M$, in a global and invariant fashion -- effectively, a Lorentzian analogue of the wave group. As already mentioned, our strategy consists in extending and adapting results from \cite{CLV}, which yield an explicit -- {\it i.e.}~up to solving ordinary differential equations -- and invariant formula for the integral kernel of the solution operator (precise definitions will be provided later on). As our technique relies heavily on microlocal analysis, all our constructions capture the singular structure of the mathematical objects involved, modulo infinitely smooth contributions.

In writing this paper we have in mind the mathematical results on the one side and their potential applications on the other. For this reason, we refrain, at times, from providing arguments and proofs in the most general possible case, in the attempt to make the paper accessible also to a mathematical physics readership not fully familiar with the technical theory of Fourier integral operators.

Our main results are as follows.

\begin{enumerate}
\item Working on an ultrastatic spacetime $\mathbb{R}\times \Sigma$ with compact Cauchy surface, we extend results from \cite{CLV} to construct the operator
\begin{equation*}
e^{-i \sqrt{-\Delta_h+V} t}
\end{equation*}
as a single oscillatory integral global in space and in time, see Subsection~\ref{Subsec: the algoritm}
 and Theorem~\ref{main theorem static propagator}. Here $\Delta_h$ is the Laplace--Beltrami operator on $\Sigma$ and $V$ a time-independent smooth potential.
 
\item Using the extended version of the results from \cite{CLV}, we provide an explicit algorithm for the construction of the wave propagator on static spacetimes, by reduction to the ultrastatic case, see Section~\ref{Sec: The wave propagator in static spacetimes}. Observe that, for this class of spacetimes, the knowledge of the wave propagator is sufficient to construct also the advanced and retarded fundamental solutions for the D'Alembert wave operator. These are two distinguished inverses characterised by their support properties. They play an important role in applications, especially in quantum field theory, since they capture the finite speed of propagation encoded within the wave equation.

\item Working on a globally hyperbolic spacetime $M$ with compact Cauchy surface and fixing a Cauchy surface $\Sigma_s:=\{s\}\times\Sigma$, we construct the operator $U(t,s):C_0^\infty(\Sigma_s) \to M$ whose integral kernel $u$ (i) satisfies the wave equation (in a distributional sense) modulo $C^\infty$, (ii) has a wavefront set of Hadamard type, as a single oscillatory integral, global in space and in time, with distinguished complex-valued phase function. See Theorem~\ref{Thm: wave propagator for globally hyperbolic spacetimes}.

\item We illustrate how our procedure can be applied to the construction of Hadamard states and we demonstrate how to recover from our global integral kernel the (local) Hadamard expansion commonly used in algebraic quantum field theory.
\end{enumerate}

We remark that G\'erard and Wrochna have also addressed the problem of constructing Hadamard states on globally hyperbolic spacetimes in \cite{Gerard-Wrochna-14}. 
Their construction shares some common features with our approach, though their work is based mainly on pseudo-differential techniques.

\vskip .5cm

\noindent The paper is structured as follows.

In Section~\ref{Geometric preliminaries} we define the geometric objects we will be using throughout the paper and fix notation and conventions.

In Section~\ref{Sec: The wave propagator in static spacetimes} we construct the wave propagator on static spacetimes. This is done in several steps. First, by a conformal transformation we reduce the problem at hand to an auxiliary problem on an ultrastatic background; secondly, working on an ultrastatic spacetime, we extend the results from \cite{CLV} to encompass the case of wave operators with arbitrary time-independent infinitely smooth potential and solve the auxiliary problem; finally, we go back to the static case by conjugating by suitable multiplication operators.

In Section~\ref{Sec: Extension to general globally hyperbolic spacetimes} we construct, using Fourier integral operator techniques, a Lorentzian analogue of the wave group. The key idea is to identify a distinguished complex-valued phase function encoding information about the propagation of singularities and the Lorentzian geometry and then exploit the construction developed in \cite{CLV} for the Riemannian case.

Lastly, in Section~\ref{Sec: An application: Hadamard states} we discuss one of the potential applications of the method, namely, the global construction of Hadamard states, relevant in algebraic quantum field theory on curved spacetime.

The paper is complemented by an appendix, Appendix~\ref{App: Fundamental Facts on Fundamental Solutions}, where we recall in an abridged manner basic facts about fundamental solutions of normally hyperbolic operators.

\section{Geometric preliminaries}
\label{Geometric preliminaries}

In this section we discuss the geometric data which will play a key role in this paper, and fix notation and conventions. 

We call \emph{spacetime} a pair $(M,g)$ where $M$ is a connected, Hausdorff, second countable, orientable smooth manifold of dimension $d\ge2$ and $g$ is a smooth Lorentzian metric of signature $(+,-,...,-)$.

A spacetime is said to be \emph{globally hyperbolic} if it possesses a \emph{Cauchy surface}, that is a closed achronal subset $\Sigma\subset M$ whose domain of dependence $D(\Sigma)$ is the whole spacetime $M$, see \cite[\S 8.1]{Wald}. Globally hyperbolic spacetimes are particularly important because their geometry ensures the well-posedness of Cauchy problems for normally hyperbolic operators ({\it e.g.}, the wave operator).

A remarkable and very convenient equivalent characterisation of global hyperbolicity was provided by Bernal and Sanchez in \cite{Bernal:2004gm,Bernal:2005qf}. We report below their result as formulated in \cite[\S 1.3]{Bar:2007zz}.

\begin{theorem}\label{Th:glob_hyp}
	Let $(M,g)$ be a time-oriented spacetime. Then the following statements are equivalent:
	\begin{enumerate}[(i)]
		\item $(M,g)$ is globally hyperbolic
		\item $(M,g)$ is isometric to $\mathbb{R}\times\Sigma$ endowed with the line element 
\begin{equation}
\label{metric general form BS}
ds^2=\beta \,dt^2-h_t,
\end{equation}
where $t:\mathbb{R}\times\Sigma\to\mathbb{R}$ is the projection onto the first factor, $\beta$ is a strictly positive smooth function on $\mathbb{R}\times\Sigma$, and $t\mapsto h_t$ is a one-parameter family of smooth Riemannian metrics. Furthermore, for all $t\in\mathbb{R}$, $\{t\}\times\Sigma$ is a smooth, $(d-1)$-dimensional spacelike Cauchy surface.
	\end{enumerate}
\end{theorem}

\noindent In this paper we will only consider globally hyperbolic spacetimes and we will always work with the metric in the standard form \eqref{metric general form BS}. Furthermore, we will make the additional assumption that $\Sigma$ is compact, hence in particular closed.
We shall adopt the convention that Greek (resp.\ Latin) indices are associated with quantities related to $\Sigma$ (resp.\ to $M$).
Moreover, points of $M$ are denoted with uppercase Latin letters, points of $\Sigma$ are denoted with lowercase Latin letters.

For later convenience, we observe that, denoting by $\pi_1:M\to\mathbb{R}$ and $\pi_2:M\to\Sigma$ the two natural projection maps,
 global hyperbolicity yields the isomorphisms of vector bundles
\begin{equation}
\label{decomposition tangent and cotangent bundle}
TM\simeq \pi_1^*T\mathbb{R}\oplus\pi_2^*T\Sigma,\qquad T^*M\simeq \pi_1^*T^*\mathbb{R}\oplus\pi_2^*T^*\Sigma,
\end{equation}
where $\oplus$ indicates the Whitney sum between the pull-back bundles, see \cite{Husemoller-94}.

For every point $Y\in M$, we denote by $\mathcal{N}(Y)$ the normal convex neighbourhood of $Y$ consisting of all points $X\in M$ such that there exists a unique geodesic $\gamma$ connecting $Y$ to $X$.
Correspondingly we define the {\em Synge's world function} as the scalar map
\begin{equation}\label{eq:Synge_world_function}
	\sigma(X,Y):=\frac{1}{2}(\lambda_1-\lambda_0)\int\limits_{\lambda_0}^{\lambda_1}d\lambda\;g(\dot{\gamma}(\lambda),\dot{\gamma}(\lambda)), \qquad \forall Y \in M, \forall X\in \mathcal{N}(Y),
\end{equation}
where $\dot{\gamma}(\lambda)$ is the tangent vector at $\gamma(\lambda)$ to the geodesic $\gamma$ connecting $X$ to $Y$ with the affine parameter chosen in such a way that $\gamma(\lambda_0)=Y$ and $\gamma(\lambda_1)=X$, see \cite[\S 3]{Poisson:2011nh}.
It is important to stress that \eqref{eq:Synge_world_function} corresponds also to half of the squared geodesic distance (in the Lorentzian sense) between $X$ and $Y$.

A globally hyperbolic spacetime is called {\em static} if it possesses an irrotational timelike Killing vector field. This is equivalent to requiring that neither $\beta$ nor $h_t$ in \eqref{metric general form BS} depend on $t$. A globally hyperbolic spacetime is called {\em ultrastatic} if it is static and, in addition, $\beta=1$. 

\begin{remark}
Observe that any $(d-1)$-dimensional closed Riemannian manifold $(\Sigma,h)$ gives rise to a unique (up to isometries) ultrastatic globally hyperbolic $d$-dimensional spacetime 
\[
(\mathbb{R}\times\Sigma,ds^2=dt^2-h),
\] 
see \cite{Kay:1978yp}. 
\end{remark}

Let $\nabla$ be the Levi-Civita connection associated with $g$ and
\begin{equation}
\label{density}
\rho_g(X):=\sqrt{\left| \det g_{ab}(X)\right|}\,
\end{equation}
be the Lorentzian density.
Given a smooth real scalar field $\Phi:M\to\mathbb{R}$, the Klein--Gordon equation is defined to be
\begin{equation}\label{Eq:KG}
P\Phi:=(\Box_g+\xi \mathcal{R}+m^2)\Phi=0,\quad \xi\in\mathbb{R},\;\textrm{and}\;m^2\geq 0,
\end{equation}
where $\mathcal{R}$ is scalar curvature  and $\Box_g$ is the D'Alembert wave operator acting on scalar functions, defined by
\begin{equation}\label{eq:d'Alembert}
\Box_g:= g^{ab}(X)\,\nabla_a\nabla_b=[\rho_g(X)]^{-1}\,\partial_a\left(\rho_g(X)\,g^{ab}(X)\,\partial_b\right).
\end{equation}
In our paper we will focus mainly on the wave equation, which is a special case of the Klein--Gordon equation obtained by setting $\xi=m=0$. 

\section{The wave propagator on static spacetimes}
\label{Sec: The wave propagator in static spacetimes}

In this section we have a twofold goal. First, we consider the d'Alembert wave operator on a globally hyperbolic static spacetime with compact Cauchy surface and we show that the construction of the advanced and retarded fundamental solutions can be reduced to that of the half-wave propagator for an associated partial differential operator. Secondly, we explicitly construct the latter, modulo an infinitely smoothing operator, by relying on results and ideas from \cite{CLV}. We start from the static scenario as it encompasses, in essence, all features of the most general case, hence providing an intuitive picture of the Lorentzian extension, stripped of all technicalities here unnecessary.

\subsection{Reduction to the ultrastatic case}
\label{Sec: Reduction to the ultrastatic case}

Consider a globally hyperbolic static spacetime $(M,\widetilde{g})$. On account of Theorem \ref{Th:glob_hyp}, $M\simeq \mathbb{R}\times \Sigma$ and the metric tensor $\widetilde{g}$ can be written as 
\[
\widetilde{g}=\beta\,dt^2 - \widetilde{h},
\]
where both $\beta$ and $\widetilde{h}$ do not depend on $t$. Via a conformal rescaling we can construct an ultrastatic metric
\[
g:=\beta^{-1}\, \widetilde{g} = \dd t^2 - h,
\]
where $h=\beta^{-1}\,\widetilde{h}$. The wave operators $\square_{\widetilde{g}}$ and $\square_g$ associated with $\widetilde{g}$ and $g$, respectively, see \eqref{eq:d'Alembert}, are related by 
\begin{equation}
\label{relation between box tilde and box}
\beta^{\frac{2+d}{4}}\,\square_{\widetilde{g}}\, \beta^{\frac{2-d}{4}}=\square_g+V,
\end{equation}
where $d=\dim M$ and
\begin{equation}
\label{pontential V static}
V=\beta^{-\frac{2-d}{4}}\left(\Delta_h \,\beta^{\frac{2-d}{4}}\right).
\end{equation}
Here and further on, $\Delta_h$ denotes the Laplace--Beltrami operator on $\Sigma$ associated with $h$. 

Direct inspection of formula \eqref{relation between box tilde and box} reveals that constructing the advanced and retarded fundamental solutions  $\widetilde{G}^\pm$ associated to $\square_{\widetilde{g}}$, {\it cf.} Theorem \ref{Thm:propagators}, is equivalent to constructing $G^\pm$, their counterparts for the operator 
\begin{equation}
\label{new hyperbolic operator}
\Box_g+V=\partial_t^2-\Delta_h+V.
\end{equation}
The two pairs of operators are related by the identity
\begin{equation}
\label{relation between adcanved and retarded static}
\widetilde{G}^\pm=\beta^{\frac{2-d}{4}}\circ G^\pm \circ \beta^{\frac{2+d}{4}}.
\end{equation}

The above argument shows that one can reduce the problem of constructing the advanced and retarded fundamental solutions for the D'Alembert wave operator on a static spacetime to the same problem for the operator \eqref{new hyperbolic operator} on an ultrastatic background, instead.

\subsection{Construction of the global static parametrix}
\label{Subsec: the algoritm}

Recently the first-named author, M.~Levitin and D.~Vassiliev have developed a geometric approach to construct the wave propagator globally and invariantly on a closed Riemannian manifold \cite{CLV}, using Fourier integral operator techniques.
We refer the reader to \cite{CLV,CV} for a detailed review of the literature on global FIOs and their applications to hyperbolic propagators.
 
What presented in \cite{CLV} is, effectively, the construction of 
advanced and retarded fundamental solutions for the D'Alembert operator on an ultrastatic spacetime, up to an operator with infinitely smooth Schwartz kernel. Dealing with an ultrastatic background is much simpler than dealing with a general globally hyperbolic spacetime, in that in the former case constructing the half-wave propagator
\begin{equation}
\label{half-wave propagator}
e^{-i t \sqrt{-\Delta}}
\end{equation}
is enough to the end of constructing the retarded-minus-advanced fundamental solution, see \cite[Section~1]{CLV}.

In the remainder of this Section, we will show how to adapt the results from \cite{CLV} to the case of static spacetimes, exploiting the reduction to an ultrastatic background presented in Subsection~\ref{Sec: Reduction to the ultrastatic case}. The extension to a wider class of globally hyperbolic spacetimes is postponed to Section~\ref{Sec: Extension to general globally hyperbolic spacetimes}.

Before doing so, let us recall in an abridged manner basic notions from symplectic geometry. We refer the reader to \cite[Vol.~IV]{Hor} for a comprehensive exposition.

Let $N$ be a smooth $(d-1)$-dimensional manifold and let $\omega \in \Omega^2(N)$ be the canonical symplectic form on the cotangent bundle $T^*N$.
A diffeomorphism $\mathcal{C}:T^*N \to T^*N$ is called a \emph{canonical transformation} if it preserves $\omega$ under pullback. In addition we say that $\mathcal{C}$ is \emph{positively homogeneous of degree $k$} if $\mathcal{C}(y,\lambda\,\eta)=\lambda^k\, \mathcal{C}(y,\eta)$ for every $\lambda>0$. The (twisted) graph $\Lambda$ of $\mathcal{C}$ is a Lagrangian submanifold of $T^*N\times T^*N$ with respect to the symplectic structure induced by $\omega$. 
Similar definitions are given when the canonical transformation depends on additional parameters. This happens, remarkably, when $\mathcal{C}$ is the Hamiltonian flow generated by a
Hamiltonian $\mathfrak{h}\in C^\infty(T'N;\mathbb{R})$,
where $T'N:=T^*N\setminus \{0\}$ denotes the cotangent bundle with the zero section removed. 

For $(y,\eta)\in T'N$, we denote by $(x^*(s;y,\eta),\xi^*(s;y,\eta))_{s\in\mathbb{R}}$ the solution to Hamilton's equations
\begin{equation*}\left\{
\begin{array}{l}
\dot{x}^*(s;y,\eta)=\mathfrak{h}_\xi((x^*(s;y,\eta),\xi^*(s;y,\eta))\,,\\
\dot\xi^*(s;y,\eta)=-\mathfrak{h}_x((x^*(s;y,\eta),\xi^*(s;y,\eta))\,,\\
(x^*(0;y,\eta),\xi^*(0;y,\eta))=(y,\eta)\,.
\end{array}\right.
\end{equation*}
Here and further on the dot stands for differentiation with respect to the flow parameter whereas the subscript $_\xi$ (resp.\ $_x$) denotes differentiation with respect to the $\xi$- (resp.\ $x$-) variable.
It is easy to see that
\[
{(y,\eta)\mapsto(x^*(s;y,\eta),\xi^*(s;y,\eta))}
\] 
is a one-parameter family of canonical transformations and, as such, it generates a one-parameter family of Lagrangian submanifolds $\Lambda_{\mathfrak{h},s}\subseteq T'N\times T'N$. If in addition the Hamiltonian $\mathfrak{h}$ is positively homogeneous of degree one in momentum, then the Lagrangian submanifold generated by the Hamiltonian flow is a conic submanifold of $T'N\times T'N$. 

\begin{definition}\label{Def: phase function (of class Lh)}
	We call \emph{phase function} a function $\varphi\in C^\infty(\mathbb{R}\times N\times T'N;\mathbb{C})$ which is
\begin{itemize}
\item \emph{nondegenerate}, {\it i.e.}
\[
\det\varphi_{x^\alpha\eta_\beta} \ne 0
\]
on 
\begin{equation}
\label{critical set}
\mathfrak{C}_\varphi:=\{(s,x;y,\eta)\,\,|\,\,\varphi_\eta(s,x;y,\eta)=0 \}\,;
\end{equation}
\item positively homogeneous of degree one in momentum, i.e.\
\[
\varphi(s,x;y,\lambda\eta)=\lambda\,\varphi(s,x;y,\eta) \qquad \forall\lambda>0, \,(s,x;y,\eta)\in \mathbb{R}\times N \times T'N.
\]
\end{itemize}	
\end{definition}

\begin{definition}	
Given a Hamiltonian $\mathfrak{h}$ positively homogeneous of degree $1$, we say that a phase function $\varphi=\varphi(s,x,y,\eta)$ is \emph{of class $\mathcal{L}_\mathfrak{h}$}, and we write $\varphi\in \mathcal{L}_\mathfrak{h}$, if 
	\begin{enumerate}[(i)]
		\item $\left.\varphi\right|_{x=x^*(s;y,\eta)}=0$,
		\item $\left.\varphi_{x^\alpha}\right|_{x=x^*(s;y,\eta)}=\xi_\alpha^*(s;y,\eta)$,
		\item $\left.\det\varphi_{x^\alpha\eta_\beta}\right|_{x=x^*(s;y,\eta)}\ne0$,
		\item $\operatorname{Im}\varphi\ge0$.
	\end{enumerate}
\end{definition}

Phase functions $\varphi\in \mathcal{L}_\mathfrak{h}$ allow one to globally parameterise the Lagrangian submanifold $\Lambda_{\mathfrak{h},s}$ \cite{LSV}, namely, in local coordinates $x$ and $y$ and in a neighbourhood of a given point of $\Lambda_{\mathfrak{h},s}$, we have
\begin{equation}
\Lambda_{\mathfrak{h},s}\simeq\{ \bigl((x,\varphi_x(s,x;y,\eta)),(y,\varphi_y(s,x;y,\eta))\bigr)\,\,|\,\,(s,x;y,\eta)\in \mathfrak{C}_\varphi \},
\end{equation}
where $\mathfrak{C}_\varphi$ is given by  \eqref{critical set}.

Specialising the above notation to the case when $N=\Sigma$ is a closed Riemannian $(d-1)$-manifold endowed with a Riemannian metric $h$, consider the ultrastatic spacetime $(M:=\mathbb{R}\times \Sigma, dt^2-h)$ and put 
\begin{equation}
\label{operator E}
E:=\Delta_h-V,
\end{equation}
where $V\in C^\infty(\Sigma;\mathbb{R})$ is a time-independent smooth real potential.

The operator $E$ is a formally self-adjoint elliptic second order linear partial differential operator on $C^\infty(\Sigma;\mathbb{C})$. As we are adding a zero order perturbation, the principal symbol of $E$ 
\begin{equation}
\label{E prin temp 1}
E_\mathrm{prin}(x,\xi):=-h^{\alpha\beta}(x)\,\xi_\alpha\,\xi_\beta, \quad (x,\xi)\in T'\Sigma,
\end{equation}
coincides with that of $\Delta_h$. 

It is well known that the singularities of the solutions to $(\partial_t^2-E)f=0$ propagate along null geodesics. In this setting, these are nothing but the lift to $M$ of the Hamiltonian flow on $\Sigma$ associated with the Hamiltonian $\mathfrak{h}(x,\xi):=\sqrt{-E_{\mathrm{prin}}(x,\xi)}$, where $t$ is equal to the parameter $s$ along the flow. 

Since $\Sigma$ is compact, the spectrum of $-E$ is discrete and accumulates to $+\infty$, but unlike for the case of the Laplace--Beltrami operator, it is no longer guaranteed to be non-negative, as the presence of the potential $V$ may bring about the appearance of negative eigenvalues. Let us denote by $\zeta_k$ the eigenvalues of $-E$ and by $v_k$ the corresponding orthonormalised eigenfunctions,
\begin{equation}
-E \,v_k=\zeta_k  \,v_k,
\end{equation}
labelling positive eigenvalues with positive index $k$ and nonpositive eigenvalues with nonpositive index $k$, in increasing order and with account of their multiplicities. 

Let 
\[
\Pi^-: L^2(\Sigma)\to L^2(\Sigma)
\]
be the orthogonal projection onto the direct sum of the eigenspaces corresponding to non-positive eigenvalues of $-E$ and put
\[
\Pi^+:=\mathrm{Id}-\Pi^-.
\]
Clearly, $\Pi^-$ has finite rank, as $E$ has at most a finite number of non-negative eigenvalues, all with finite multiplicity.

Let us define
\begin{equation}
\label{operator U static with functional calculus}
U^+(t):=\sum_{\zeta_k>0} e^{-it\sqrt{\zeta_k}} \, v_k \,(v_k, \,\cdot\,)_{L^2(\Sigma)},
\end{equation}
where $\sqrt{\zeta_k}$ is the positive square root of $\zeta_k>0$.
One can see that, for any choice of the square root $\sqrt{-E}$ (with branch cut away from the spectrum) such that
\[
U^+(t)=e^{-it\sqrt{-E}}\,\Pi^+,
\]
the operator
\[
U^+(t)-e^{-it\sqrt{-E}}= e^{-it\sqrt{-E}}\,\Pi^-
\]
is infinitely smoothing.
\

This tells us that constructing $U^+(t)$ is equivalent to constructing $e^{-it\sqrt{-E}}$ up to an infinitely smoothing operator and, furthermore, that working modulo operators with infinitely smooth kernel allows one to disregard the ambiguity in the definition of $\sqrt{-E}$ on $\Pi^- L^2(\Sigma)$ and in the choice of the branch cut.

Define $L^2_\pm(\Sigma):=\Pi^\pm L^2(\Sigma)$. 
The operator $U^+(t)$ satisfies, in a distributional sense,
\begin{subequations}
\label{approximation U properties}
\begin{equation}
(-i \partial_t +\sqrt{-E})U^+(t)=0,
\end{equation}
\begin{equation}
\label{Identity U plus exact}
U^+(0)|_{L^2_+(\Sigma)}=\mathrm{Id}_{L^2_+(\Sigma)}.
\end{equation}
\end{subequations}
Our strategy consists in approximating the operator
\begin{equation}
\label{operator U without plus static}
U(t):=U^+(t)|_{L^2_+(\Sigma)}\oplus \mathrm{Id}_{L^2_-(\Sigma)}:L^2_+(\Sigma)\oplus L^2_-(\Sigma)\to L^2(\Sigma).
\end{equation}
by a single oscillatory integral, global in space and in time. The operator $U(t)$ will 
\begin{enumerate}[(a)]
	\item
	satisfy equations \eqref{approximation U properties} on $L^2(\Sigma)$ up to an infinitely smoothing operator and

	\item
	coincide with $U^+(t)$, hence with $e^{-it\sqrt{-E}}$ up to an infinitely smoothing operator.
\end{enumerate}

\begin{theorem}
\label{main theorem static propagator}
The Schwartz kernel 
\begin{equation}
\label{integral kernel static}
u(t,x,y)=\sum_{\zeta_k>0} e^{-it\sqrt{\zeta_k}} \, v_k(x)  \overline{v_k(y)}+\sum_{\zeta_k\le 0} v_k(x)  \overline{v_k(y)}
\end{equation}
of the operator \eqref{operator U without plus static} can be written, modulo an infinitely smooth function in all variables, as a single oscillatory integral
\begin{equation}
\label{oscillatory integral static theorem}
u(t,x,y)\overset{\mod C^\infty}{=}\dfrac{1}{(2\pi)^{d-1}}\int_{T^*_y\Sigma} e^{i \,\varphi(t,x;y,\eta)}\,\mathfrak{a}(t;y,\eta)\, \chi_\mathfrak{h}(t,x;y,\eta) \,w(t,x;y,\eta)\,d \eta,
\end{equation}
global in space and in time, where
\begin{enumerate}[(a)]
\item $\varphi$ is a phase function of class $\mathcal{L}_\mathfrak{h}$, with $\mathfrak{h}(y,\eta)=\sqrt{h^{\alpha\beta}(y)\,\eta_\alpha\eta_\beta}$;
\item $\mathfrak{a}\in C^\infty(\mathbb{R},S^0_{\mathrm{ph}}(T'\Sigma))$, namely it is a polyhomogeneous symbol of order zero depending smoothly on $t$,
\[
\mathfrak{a} \sim \sum_{j=0}^\infty \mathfrak{a}_{-j}, \qquad 
\mathfrak{a}_{-j}(t;y,\lambda\,\eta)=\lambda^{-j}\, \mathfrak{a}_{-j}(t;y,\eta)\quad \forall\lambda>0;
\]
\item $\chi_\mathfrak{h}$ is a smooth cut-off subordinated to $\mathfrak{h}$, namely, an infinitely smooth function on $\mathbb{R}\times \Sigma\times T'\Sigma$ satisfying
	\begin{enumerate}[(i)]		
			\item $\chi_\mathfrak{h}(t,x;y,\eta)=0$ on $\{(t,x;y,\eta) \,|\, |\mathfrak{h}(y,\eta)|\leq 1/2\}$;
			\item $\chi_\mathfrak{h}(t,x;y,\eta)=1$ on the intersection between $\{(t,x;y,\eta) \,|\, |\mathfrak{h}(y,\eta)| \geq 1\}$ and any conical neighbourhood of $\{(t,x^\ast(t;y,\eta);y,\eta) \}$;
			\item $\chi_\mathfrak{h}(t,x;y,\alpha\, \eta)=\chi_\mathfrak{h}(t,x;y,\eta)$ for $\alpha\geq 1$ on $\{ (t,x;y,\eta) \, | \, |\mathfrak{h}(y,\eta)|\geq 1  \}$;
	\end{enumerate}
	
\item the weight $w$ is defined as
\begin{equation}
\label{definition weight}
w(t,x;y,\eta):=
[\rho_h(x)]^{-1/2}\,[\rho_h(y)]^{-1/2}
\left[
{\det}^2\!
\left(
\varphi_{x^\alpha \eta_\beta}(t,x;y,\eta)
\right)
\right]^{1/4},
\end{equation}
where $\rho_h$ is the Riemannian density ({\it cf.}\ \eqref{density}) and the branch of the complex root is chosen so that 
\[
\left.
\arg \left[
{\det}^2\!
\left(
\varphi_{x^\alpha \eta_\beta}(t,x;y,\eta)
\right)
\right]^{1/4} \right|_{t=0} =0.
\]
\end{enumerate}
\end{theorem}

\begin{remark}
Let us point out that the weight \eqref{definition weight} is a smooth scalar function in the variables $t$, $x$ and $\eta$ and a smooth $(-1)$-density in the variable $y$. The choice of powers of the Riemannian density serves the purpose of making the integral kernel \eqref{oscillatory integral static theorem} a scalar function in all variables. The weight $w$ is a crucial element in our construction in that it ensures the right covariance properties of \eqref{oscillatory integral static theorem} and its nondegeneracy -- a consequence of using complex-valued phase functions of class $\mathcal{L}_\mathfrak{h}$ -- is key to performing a construction global in time, circumventing topological obstructions offered by caustics.
\end{remark}

\begin{proof}[Proof of Theorem~\ref{main theorem static propagator}]
Denote by
\[
\mathcal{F}_{\lambda\to t}[f](t)=\widehat{f}(t)=\int_{-\infty}^{+\infty}e^{-it\lambda}f(\lambda)\,d\lambda
\]
the Fourier transform and by
\[
\mathcal{F}^{-1}_{t\to \lambda}[\widehat f](t):=f(\lambda)=\frac{1}{2\pi}\int_{-\infty}^{+\infty}e^{it\lambda}\widehat{f}(t)\,dt
\]
the inverse Fourier transform.

Assume one has constructed $u\in C^\infty(\mathbb{R}; \mathcal{D}'(\Sigma_x\times \Sigma_y))$ such that 
\begin{enumerate}
\item $
(\partial_t^2-E^{(x)})u(t,x,y)=0  \mod C^\infty,
$
where the superscript $^{(x)}$ indicates that $E$ acts in the variable $x$;

\item $u$ satisfies
\[
\int_\Sigma u(0,x,y)\,\left(\,\cdot\,\right)\,\rho(y)\,dy=\mathrm{Id} \qquad \mod C^\infty;
\]

\item for every $\psi\in C^\infty_0(\mathbb{R})$,
\[
\mathcal{F}^{-1}_{t\to \lambda}[\psi(t)u(t,x,y)]=O(|\lambda|^{-\infty}) \qquad \text{as }\lambda\to-\infty.
\]
\end{enumerate}
How to do so will be explained in Subsection~\ref{the algorithm}. Then our theorem follows from an adaptation of results from \cite{CLV}.
\end{proof}

\begin{corollary}
Let $M$ be a static spacetime. Then, the integral kernel of the operators $\mathcal{G}_{0,\overline{t}}$ and $\mathcal{G}_{1,\overline{t}}$ defined in accordance with Proposition~\ref{Prop: initial data to solutions} can be represented, modulo $C^\infty$, as the sum of two oscillatory integrals of the form \eqref{oscillatory integral static theorem}, global in space and time.
\end{corollary}

\begin{proof}
Let $E$ be given by \eqref{operator E} for the choice of $V$ \eqref{pontential V static}. The operators $\mathcal{G}_{0,\overline{t}}$ and $\mathcal{G}_{1,\overline{t}}$ associated with the normally hyperbolic operator $\partial_t^2-E$ on the ultrastatic spacetime $\mathbb{R}\times \Sigma$ are given by
\[
\mathcal{G}_{0,\overline{t}}=\cos(\sqrt{-E}(t-\overline{t})),\quad\mathcal{G}_{1,\overline{t}}=\sin(\sqrt{-E}(t-\overline{t}))(\sqrt{-E})^{-1}+t\,\Pi^0,
\]
where $\Pi^0$ is the projection onto the kernel of $E$, see \cite[Section~1]{CLV} and \cite[Chap. 3]{Fulling}. Hence, the statement follows from Theorem~\ref{main theorem static propagator} combined with \eqref{relation between adcanved and retarded static}.
\end{proof}

Observe that, unlike in the construction of the half-wave propagator, there is no ambiguity in defining $\mathcal{G}_{0,\overline{t}}$ and $\mathcal{G}_{1,\overline{t}}$ by means of spectral calculus even if the spectrum of $-E$ has non-empty intersection with the negative real line. The reason lies in the fact that the functions $\cos(x)$ and $\frac{\sin(x)}{x}$ are even.

\subsubsection{The algorithm}
\label{the algorithm}

In the remainder of this subsection, we will explain how to construct the oscillatory integral in the RHS of \eqref{oscillatory integral static theorem}, adapting \cite[Section~5]{CLV} to the case in hand. 

The starting main idea consists in fixing a distinguished phase function and setting out an algorithm for the construction of $\mathfrak{a}$ in a global, invariant fashion. 

\begin{definition}[Levi-Civita phase function {\cite[Definition~4.1]{CLV}}]
We define the \emph{Levi-Civita phase function} to be
\begin{equation}
\label{Eq: Levi-Civita phase function ultrastatic}
\varphi_\epsilon(t,x;y,\eta):=-\frac{1}{2} \langle{\xi^*(t;y,\eta)},{\left. \operatorname{grad}_z [\operatorname{dist}_\Sigma^2(x,z)]\right|_{z=x^*(t;y,\eta)}}\rangle
+\dfrac{i \, \epsilon }{2} \,\mathfrak{h}(y,\eta) \,
\operatorname{dist}_\Sigma^2(x,x^\ast(t;y,\eta))
\end{equation}
when $x$ lies in a $\Sigma$-geodesic neighbourhood of $x^*(t;y,\eta)$ and continued smoothly and arbitrarily elsewhere, in such a way that $\mathrm{Im}(\varphi)\ge 0$. Here $\operatorname{dist}_\Sigma$ is the (Riemannian) geodesic distance on $\Sigma$, $\langle \,\cdot\,,\,\cdot\,\rangle=h^{-1}(\,\cdot\,,\,\cdot\,)$ and $\epsilon>0$ is a positive parameter pre-multiplying the imaginary part of $\varphi$.
\end{definition}

\noindent The symbol $\mathfrak{a}$ is determined as follows.

\vskip .2cm

\textbf{Step 1}. Set $\chi\equiv 1$ --- contributions from regions where $\chi\ne1$ are smooth, as one can establish via a stationary phase argument --- and act on the oscillatory integral \eqref{oscillatory integral static theorem} with the operator $\partial_t^2-E^{(x)}$. This yields a new oscillatory integral
\begin{equation}
I_\varphi(a):=\dfrac{1}{(2\pi)^{d-1}}\int_{T^*_y\Sigma} e^{i \,\varphi(t,x;y,\eta)}\,a(t,x;y,\eta) \,w(t,x;y,\eta)\,d \eta
\end{equation}
which has the same form of  \eqref{oscillatory integral static theorem} though with a new amplitude
\begin{equation}\label{Eq:a-amplitude}
a(t,x;y,\eta):=\dfrac{e^{-i\,\varphi}}{w} \,(\partial_t^2-E) \left[e^{i \,\varphi}\, \mathfrak{a}\, w\right].
\end{equation}
Observe that the amplitude $a\in C^\infty(\mathbb{R}\times \Sigma; S^2_{\mathrm{ph}}(T'\Sigma))$ depends on the variable $x$.

\vskip .2cm

\textbf{Step 2}. As a next step, we exclude the dependence on $x$ of the amplitude $a$, namely, we find $\mathfrak{b}\in C^\infty(\mathbb{R};S^2_{\mathrm{ph}}(T^\prime \Sigma))$ such that $I_\varphi(a)=I_\varphi(\mathfrak{b})$ mod $C^\infty$. 
This is achieved by the so-called \emph{reduction of the amplitude}, which exploits suitably devised differential-evaluation operators. For every $\alpha=1,...,d-1$, put
\begin{gather}
L_\alpha:C^\infty(\mathbb{R}\times \Sigma;S^j_{\mathrm{ph}}(T^\prime\Sigma))\to C^\infty(\mathbb{R}\times \Sigma;(S^j_{\mathrm{ph}}T^\prime\Sigma)),\notag\\
L_\alpha:=\left[(\varphi_{x\eta})^{-1}\right]_\alpha{}^\beta\,\dfrac{\partial}{\partial x^\beta},\label{operator L}
\end{gather}
where $\varphi_{x\eta}^{-1}$ is defined in accordance with
\[
[\varphi^{-1}_{x\eta}]_\alpha{}^\beta\,\varphi_{x^\beta\eta_\gamma}=\delta_\alpha{}^\gamma.
\]
In addition, for all $k\in\mathbb{N}_0$, let $\mathfrak{S}_{-k}:C^\infty(\mathbb{R}\times \Sigma;S^j_{\mathrm{ph}}(T^\prime\Sigma))\to C^\infty(\mathbb{R};S^{j-k}_{\mathrm{ph}}(T^\prime\Sigma))$ be defined as
\begin{subequations}\label{operators mathfrak S}
\begin{gather}
\mathfrak{S}_0 f:=f|_{x=x^*(t;y,\eta)}\,, \label{mathfrak S0}\\
\mathfrak{S}_{-k}f:=\mathfrak{S}_0 \left[ i \, w^{-1} \frac{\partial}{\partial \eta_\beta}\, w \left( 1+ \sum_{1\leq |\boldsymbol{\alpha}|\leq 2k-1} \dfrac{(-\varphi_\eta)^{\boldsymbol{\alpha}}}{\boldsymbol{\alpha}!\,(|\boldsymbol{\alpha}|+1)}\,L_{\boldsymbol{\alpha}} \right) L_\beta  \right]^k f\,, \quad k>0.\label{mathfrak Sk}
\end{gather}
\end{subequations}
Bold Greek letters in \eqref{mathfrak Sk} denote multi-indices in $\mathbb{N}^{n-1}_0$, $\boldsymbol{\alpha}=(\alpha_1, \ldots, \alpha_{n-1})$, $|\boldsymbol{\alpha}|=\sum_{j=1}^{n-1} \alpha_j$ and $(-\varphi_\eta)^{\boldsymbol{\alpha}}:=(-1)^{|\boldsymbol{\alpha}|}\, (\varphi_{\eta_1})^{\alpha_1}\dots ( \varphi_{\eta_{n-1}})^{\alpha_{n-1}}$. All derivatives act on whatever is to the right, unless otherwise specified. The operator
\eqref{mathfrak Sk} is well defined,
because the differential operators $L_\alpha$ commute, {\it cf.} \cite[Lemma A.2]{CLV}.

Denoting by
$
a \sim \sum_{j=0}^\infty a_{2-j}
$
the asymptotic polyhomogeneous expansion of $a$, we construct a polyhomogeneous symbol $\mathfrak{b}$ whose homogeneous components are defined as
\begin{equation}\label{construction homogeneous components of the reduced amplitude}
\mathfrak{b}_l:=\sum_{2-j-k=l} \mathfrak{S}_{-k}\,a_{2-j},\qquad
l=2,1,0,-1,\ldots.
\end{equation}
This gives \cite[Appendix A]{CLV}
\begin{equation}
\dfrac{1}{(2\pi)^{d-1}}\int_{T^*_y\Sigma} e^{i \,\varphi}\,a \,w\,d \eta=\dfrac{1}{(2\pi)^{d-1}}\int_{T^*_y\Sigma} e^{i \,\varphi}\,\mathfrak{b} \,w\,d \eta\,\,\mod C^\infty.
\end{equation}
We adopt the terminology from \cite{CLV} and call the operator $\mathfrak{S} \sim \sum_{k=0}^{\infty} \mathfrak{S}_{-k}\,$ \emph{amplitude-to-symbol operator}. 
Clearly, the symbol $\mathfrak{b}$ is $x$-independent.

\vskip .2cm

\textbf{Step 3.} Impose
$$
(\partial_t-E^{(x)})u=I_\varphi(\mathfrak{b})=0 \mod C^\infty.
$$
In view of equations \eqref{Eq:a-amplitude} and \eqref{construction homogeneous components of the reduced amplitude}, this amounts to equating each asymptotic homogeneous component of $\mathfrak{b}$ to zero:
\begin{equation}\label{ultrastatic transport equations}
\mathfrak{b}_{l}(t;y,\eta)=0,\quad l=2,1,0,-1,....
\end{equation}
This yields a hierarchy of transport equations -- ordinary differential equations in the variable $t$ -- whose unknowns are the (scalar) homogeneous components $\mathfrak{a}_{-k}$ of $\mathfrak{a}$. 

Initial conditions $\mathfrak{a}_{-k}(0;y,\eta)$ for \eqref{ultrastatic transport equations} are established by requiring that $u(0,x,y)$ is, modulo a smooth function, the integral kernel of the identity operator.

Note that the identity operator appears in our construction as an invariant pseudodifferential operator with integral kernel
\begin{equation}
\label{identity static}
\int_{T^*_y\Sigma} e^{i \,\phi(x;y,\eta)}\,\mathfrak{s}(y,\eta)\,\chi_\mathfrak{h}(0,x;y,\eta)\,\left[ {\det} ^2 \phi_{x\eta} \right]^{1/4}\,\left(\dfrac{\rho_h(y)}{\rho_h(x)}\right)^{1/2}\,d\eta \mod C^\infty,
\end{equation}
where $\phi(x;y,\eta):=\varphi(0,x;y,\eta)$ and $\mathfrak{s}\in S^0_{\mathrm{ph}}(T'\Sigma)$, so that the initial conditions read
\[
\mathfrak{a}_{-k}(0;y,\eta)=\mathfrak{s}_{-k}(y,\eta).
\]
We refer the reader to \cite[Section~6]{CLV} for a detailed analysis of the operator \eqref{identity static}.

\

\noindent The above algorithm identifies uniquely and invariantly an element $u\in C^\infty(\mathbb{R}; \mathcal{D}'(\Sigma_x\times \Sigma_y))$  satisfying conditions 1.--3. in the proof of Theorem~\ref{main theorem static propagator}.

\section{Extension to general globally hyperbolic spacetimes}
\label{Sec: Extension to general globally hyperbolic spacetimes}

In this section we generalise the construction of the global propagator with FIO methods to the case of a general globally hyperbolic spacetime $(M,g)$ with compact Cauchy surface.

For a general global hyperbolic spacetime, the strategy presented in Section~\ref{Sec: The wave propagator in static spacetimes} does not work any more, because space and time intertwine in an essential way,  and one cannot identify a time-independent analogue of the operator $E$. In order to construct the wave propagator, one needs to abandon the operator \eqref{operator U without plus static} and adopt a more abstract approach.


The first crucial step towards our goal rests in the identification of a Lorentzian analogue of the Levi-Civita phase function, {\it cf.}~Section~\ref{Subsec: the algoritm}.

\subsection{The Lorentzian Levi-Civita phase function}

Let $(M,g)$ be a globally hyperbolic spacetime with compact Cauchy surface of dimension $d \ge 3$, realised as $\mathbb{R}\times\Sigma$, {\it cf.} Theorem \ref{Th:glob_hyp}, and let $\Sigma_s\simeq \{s\}\times \Sigma$ be an arbitrary but fixed Cauchy surface.

\begin{notation}
\label{notation on momenta}
Let $Y=(s,y)\in M$ and let $\iota_s:\Sigma_s \to M$ be the embedding of $\Sigma_s$ into $M$. Given $\eta\in T^*_y\Sigma$ we denote by $\eta_+$ the unique future pointing null covector in $T^*_YM$ such that $\iota^*_s \eta_+=\eta$. Furthermore, we put $\widehat{\eta}_+:=\frac{\eta_+}{\|\eta\|_{\Sigma_s}}$, where $\|\eta\|_{\Sigma_s}=\sqrt{h_s^{\alpha\beta}(y)\,\eta_\alpha\eta_\beta}$ and $h_s:=\iota_{\Sigma_s}^*g$.
\end{notation}

\begin{definition}
Let $Y=(s,y)\in M$. We call \emph{Levi-Civita flow} the map
\begin{align}\label{Eqn: Lorentzian flow}
	\varsigma\mapsto(\widetilde{X}^*(\varsigma;s,y,\eta),\widetilde{\Xi}^*(\varsigma;s,y,\eta)),
\end{align}
where
\begin{itemize}
\item
$\widetilde{X}^*(\,\cdot\,;s,y,\eta): \varsigma \mapsto \widetilde{X}^*(\varsigma;s,y,\eta)$ is the unique null geodesic stemming from $Y$ with initial cotangent vector $\widehat{\eta}_+$, parameterised by proper time;

\item
$\widetilde{\Xi}^*(\varsigma;s,y,\eta)$ is the parallel transport along $\widetilde{X}^*(\,\cdot\,;s,y,\eta)$ of $\eta_+$, from $Y$ to $\widetilde{X}^*(\varsigma;s,y,\eta)$.
\end{itemize}
\end{definition}

The Levi-Civita flow satisfies, by construction,
\[
(\widetilde{X}^*(s;s,y,\eta), \widetilde{\Xi}^*(s;s,y,\eta))=(Y,\eta_+)
\]
and
\begin{equation}
\label{homogeneity Levi-Civita flow}
(\widetilde{X}^*(\varsigma;s,y,\lambda\,\eta),\widetilde{\Xi}^*(\varsigma;s,y,\lambda\,\eta))=(\widetilde{X}^*(\varsigma;s,y,\eta),\lambda\,\widetilde{\Xi}^*(\varsigma;s,y,\eta)),
\end{equation}
for every $\lambda>0$.

We have the following standard result.

\begin{lemma}\label{Lem: reparametrization of Lorentzian flow}
	The Levi-Civita flow can be parameterised using the global time coordinate $t$ defined in accordance with Theorem~\ref{Th:glob_hyp} (ii).
\end{lemma}
\begin{proof}
	Let $\widetilde{X}^*(\varsigma;s,y,\eta)=(\tau^*(\varsigma;s,y,\eta),x^*(\varsigma;s,y,\eta))$, where 
	\[
	\tau^*(\varsigma;s,y,\eta):=t(\widetilde{X}^*(\varsigma;s,y,\eta))
	\] and 
	\[
	x^*(\varsigma;s,y,\eta):=\pi_{\Sigma_{\tau^*(\varsigma;s,y,\eta)}}(\widetilde{X}^*(\varsigma;s,y,\eta))
	\]
	is the projection of $\widetilde{X}^*(\varsigma;s,y,\eta)$ onto $\Sigma_{\tau^*(\varsigma;s,y,\eta)}$.
	Denoting by $\frac{\mathrm{d}\widetilde{X}^*}{\mathrm{d}\varsigma}$ the tangent vector to $\widetilde{X}^*(\,\cdot\,;s,y,\eta)$, we have $g(\frac{\mathrm{d}\widetilde{X}^*}{\mathrm{d}\varsigma},\frac{\mathrm{d}\widetilde{X}^*}{\mathrm{d}\varsigma})=g^{-1}(\widehat{\eta}_+,\widehat{\eta}_+)=0$.
	Since $\eta\in T'\Sigma$, this implies that $\frac{\mathrm{d}\tau^*}{\mathrm{d}\varsigma}$ cannot vanish.
\end{proof}

In view of Lemma~\ref{Lem: reparametrization of Lorentzian flow}, in what follows we shall denote with $({X}^*(t;s,y,\eta), {\Xi}^*(t;s,y,\eta))$ the Levi-Civita flow parameterised by $t$.

We are now in a position to define our Lorentzian phase function.

\begin{definition}\label{Def: phase function for Lorentzian manifolds}
	Let $\epsilon>0$ be a positive parameter, let $X=(t,x), Y=(s,y)\in M$ and let $(X^*(t;s,y,\eta)\,,\,\Xi^*(t;s,y,\eta))$ be the Levi-Civita flow.
	We define \emph{the Lorentzian Levi-Civita phase function} to be the infinitely smooth function $\varphi: M\times \mathbb{R}\times T'\Sigma \to \mathbb{C}$ defined by
	\begin{multline}\label{Eq: phase function for Lorentzian manifolds}
		\varphi(\tau,x;s,y,\eta): =
		- \langle\Xi^*(\tau;s,y,\eta)\,,\,\operatorname{grad}_Z \sigma(X,Z)|_{Z=X^*(\tau;s,y,\eta)}  \rangle\\ +
		{i \,\epsilon}\|\eta\|_{\Sigma_s}\sigma(X,X^*(\tau;s,y,\eta))\,,
	\end{multline}
	if $X$ lies in a geodesic normal neighbourhood of  $X^*(\tau;s,y,\eta)$, and smoothly continued elsewhere in such a way that the imaginary part is positive.
	Here $\sigma$ is the Ruse--Synge world function, defined in accordance with \eqref{eq:Synge_world_function}.
\end{definition}
\begin{remark}
	Observe that the Lorentzian Levi-Civita phase function $\varphi$ can be equivalently recast as
	\begin{multline}\label{Eqn: Lorentzian phase function}
	\varphi(\tau,x;s,y,\eta)
	=- \frac12 \langle \iota_{\Sigma_\tau}^*\Xi^*(\tau;s,y,\eta)\,,\,\left. \operatorname{grad}_z(\operatorname{dist}_{\Sigma_\tau}(x,z)^2)\right|_{z=x^*(\tau;s,y,\eta)}  \rangle\\+
	\dfrac{i \,\epsilon}{2}\|\eta\|_{\Sigma_s}\operatorname{dist}^2_{\Sigma_\tau}(x,x^*(\tau;s,y,\eta)),
	\end{multline}
	where $\operatorname{dist}_{\Sigma_\tau}$ is the (Riemannian) geodesic distance on $\Sigma_\tau:=\{\tau\}\times \Sigma\simeq\Sigma$.
	Whenever $(M,g)$ is ultrastatic the latter equation coincides with equation \eqref{Eq: Levi-Civita phase function ultrastatic} -- in which case $\|\eta\|_{\Sigma_\tau}=\mathfrak{h}(y,\eta)$ for every $\tau$.
\end{remark}
The following proposition tells us that Definition~\ref{Def: phase function for Lorentzian manifolds} identifies a function $\varphi$ with essentially the same properties of a phase function of class $\mathcal{L}_{\mathfrak{h}}$ -- \textit{cf.} Definition~\ref{Def: phase function (of class Lh)}.

\begin{proposition}
	The Lorentzian Levi-Civita phase function \eqref{Eq: phase function for Lorentzian manifolds} satisfies the following properties:
	\begin{enumerate}[(a)]
		\item
		The function $\varphi$ is positively homogeneous of degree one in momentum, {\it i.e.} 
		\[
		\varphi(\tau,x;s,y,\lambda\eta)=\lambda\,\varphi(\tau,x;s,y,\eta),\qquad \forall\lambda>0.
		\]
		\item
		$\operatorname{Im}\varphi\geq 0$ and, moreover,
		\begin{subequations}\label{Eqn: properties of the phase function}
			\begin{align}
				\label{Eqn: phi vanishes on the flow}
				\varphi(\tau,x^*(\tau;s,y,\eta);s,y,\eta)&=0\,,\\
				\label{Eqn: d phi is the cotangent component on the flow}
				\varphi_{x^\alpha}(\tau,x^*(\tau;s,y,\eta);s,y,\eta)&=\xi_\alpha^*(\tau;s,y,\eta)\,,\\
				\label{Eqn: phi is non-degenerate}
				\det\varphi_{x^\alpha\eta_\beta}(\tau,x^*(\tau;s,y,\eta);s,y,\eta)&\neq 0\,,
			\end{align}
			where $\xi^*:=\iota_{\Sigma_s}\Xi^*$.
		\end{subequations}
			\item
			If we define
			\begin{align*}
				\Phi&:=\{(t,x^*(t;s,y,\eta);s,y,\eta)\;|\; t,s\in\mathbb{R},(y,\eta)\in T^*\Sigma\}\,,\\
				\mathfrak{C}_\varphi&:=\{(\tau,x;s,y,\eta)\in M\times\mathbb{R}\times T'\Sigma\,\,|\,\,\varphi_\eta(s,x;y,\eta)=0 \}\,,
			\end{align*}
			then $\Phi\subseteq\mathfrak{C}_\varphi$.
			Furthermore, there exists a neighbourhood $W$ of $\Phi$ such that $\mathfrak{C}_\varphi\cap(W\setminus\Phi)=\emptyset$.
	\end{enumerate}
\end{proposition}

\begin{proof}
(a) The function $\varphi$ is positively homogeneous in $\eta$ of degree $1$ because $X^*$ and $\Xi^*$ are positively homogeneous of degree $0$ and $1$, respectively -- \textit{cf.} equation \eqref{homogeneity Levi-Civita flow}.

\

(b) The inequality $\operatorname{Im}\varphi\geq 0$ follows at once from \eqref{Eqn: Lorentzian phase function}. So let us prove \eqref{Eqn: properties of the phase function}.
Recall the following well-known identities for the Ruse--Synge world function $\sigma$ \cite[Sec. 4.1]{Poisson:2011nh}:
	\begin{subequations}\label{Eq: properties of world function}
		\begin{align}
		\label{Eq: properties of world function a}
			\sigma(X,X)&=0\,,\\
			\label{Eq: properties of world function b}
			\mathrm{d}\sigma(X,X)&=0\,,\\
			\label{Eq: properties of world function c}
			[\nabla_a\nabla'_b\sigma(X,X')]|_{X=X'}&=
			[\nabla_a\nabla_b\sigma(X,X')]|_{X=X'}=-g_{ab}(X)\,,\\
			\label{Eq: properties of world function d}
			[\nabla_a\nabla'_b\nabla'_c\sigma(X,X')]|_{X=X'}&=0\,.
		\end{align}
	\end{subequations}
	Here $\nabla'$ indicates that the Levi-Civita connection $\nabla$ on $(M,g)$ acts on the second argument of $\sigma$.
	
	Since $\mathrm{d}\sigma(X,X)=0$, \eqref{Eqn: phi vanishes on the flow} is satisfied.
	Moreover, \eqref{Eq: properties of world function b} and \eqref{Eq: properties of world function c} ensure that \eqref{Eqn: d phi is the cotangent component on the flow} holds. Indeed,
	\begin{align*}
		\varphi_{x^\alpha}=
		\nabla_\alpha\varphi=
		-g^{\beta\nu}\nabla_{\alpha}\nabla'_{\beta}\sigma\;\Xi^*_\nu+
		i\epsilon\|\eta\|_{\Sigma_s}\nabla_\alpha\sigma\stackrel{x=x^*(t;s,y,\eta)}{=}
		\xi^*_\alpha\,.
	\end{align*}
	Finally let us prove that \eqref{Eqn: phi is non-degenerate}.
	In view of \cite[Cor. 2.4.5]{SaVa}, it is enough to check that $ \operatorname{Im}(\varphi_{x^\alpha x^\beta}(\tau,x^*(\tau;s,y,\eta),s,y,\eta))>0$ as a bilinear form.
	A direct calculation shows that
	\begin{align*}
		\frac1\epsilon\operatorname{Im}\varphi_{x^\alpha x^\beta}=
		\|\eta\|_{\Sigma_s}\partial_{x^\beta}\nabla_{x^\alpha}\sigma=
		\|\eta\|_{\Sigma_s}\nabla_{x^\beta}\nabla_{x^\alpha}\sigma
		+\|\eta\|_{\Sigma_s}\Gamma_{\beta\alpha}^{\phantom{\beta\alpha}{a}}\nabla_{x^a}\sigma\,.
	\end{align*}
	The above equation, combined with \eqref{Eq: properties of world function a}-- \eqref{Eq: properties of world function d}, implies that 
\[
\operatorname{Im}(\varphi_{x^\alpha x^\beta}(\tau,x^*(\tau;s,y,\eta),s,y,\eta))=\epsilon\,\|\eta\|_{\Sigma_s}h_{\alpha\beta},
\] 
which is positive definite.
	
\

(c)
	Differentiating \eqref{Eqn: phi vanishes on the flow} with respect to $\eta_\beta$ we get
	\begin{align*}
		0&=\partial_{\eta_\beta}[\varphi(t,x^*(t;s,y,\eta);s,y,\eta)]
		\\&=\varphi_{\eta_\beta}(t,x^*(t;s,y,\eta);s,y,\eta)
		+\varphi_{x^\alpha}(t,x^*(t;s,y,\eta);s,y,\eta)(x^*)^\alpha_{\eta_\beta}(t;s,y,\eta)
		\\&=\varphi_{\eta_\beta}(t,x^*(t;s,y,\eta);s,y,\eta)
		+\xi^*_\alpha(t,x^*(t;s,y,\eta);s,y,\eta)(x^*)^\alpha_{\eta_\beta}(t;s,y,\eta)\,,
	\end{align*}
	where we used equation \eqref{Eqn: d phi is the cotangent component on the flow}.
	The task at hand is to show that
	\begin{align*}
		\xi^*_\alpha(t,x^*(t;s,y,\eta);s,y,\eta)(x^*)^\alpha_{\eta_\beta}(t;s,y,\eta)=0\,,
	\end{align*}
	from which $\Phi\subseteq\mathfrak{C}_\varphi$ descends.
	For that, we observe that $\Xi^*(t;s,y,\eta)=\|\eta\|_{\Sigma_s}\,\widehat{\Xi}^*(t,s,y,\eta)$, where $\widehat{\Xi}^*(t;s,y,\eta)$ is the cotangent vector to the geodesic $\gamma\colon t\mapsto X^*(t;s,y,\eta)$,
\[
\widehat{\Xi}^*_a=g_{ab}\frac{\mathrm{d}(X^*)^b}{\mathrm{d}t}.
\]
	Being the geodesic flow Hamiltonian, it follows that the canonical $1$-form $\Theta=\Xi_a\mathrm{d}X^a$ is conserved along the geodesic.
	Observing that $\Xi^*$ differs from $\widehat{\Xi}^*$ by a constant factor, we conclude that
	\begin{align*}
		(\eta_+)_a\,\mathrm{d}Y^a
		=\Xi^*_a\,\mathrm{d}(X^*)^a
		=\Xi^*_a\,\big[(X^*)^a_{Y^b}\mathrm{d}Y^b+(X^*)^a_{\eta_b}\mathrm{d}\eta_b\big].
	\end{align*}
	The latter equation implies $\Xi^*_a(X^*)^a_{Y^b}=(\eta_+)_b$ as well as $0=\Xi^*_a(X^*)^a_{\eta_b}=\xi^*_\alpha(x^*)^\alpha_{\eta_b}$, where we have used the fact that $X^*(t;s,y,\eta)=(t,x^*(t,s,y,\eta))$ -- \textit{cf.} Lemma \ref{Lem: reparametrization of Lorentzian flow}.
	
	Finally, let us assume by contradiction that for all neighbourhoods $W$ of $\Phi$ we have $\mathfrak{C}_\varphi\cap(W\setminus\Phi)\neq\emptyset$.
	Then, there exists a sequence $(t_n,x_n;s_n,y_n,\eta_n)\in\mathfrak{C}_\varphi$ converging to $(t,x^*;s,y,\eta)\in\Phi$. A standard Taylor expansion argument shows that this implies 
	\[
	\det\varphi_{x^\alpha\eta_\beta}(t,x^*;s,y,\eta)=0,
	\]
	which contradicts \eqref{Eqn: phi is non-degenerate}.
\end{proof}

\subsection{The Lorentzian parametrix}

The Lorentzian Levi-Civita phase function can be exploited to construct a suitable Lorentzian analogue of the operator \eqref{operator U without plus static}.

%

Let $\Sigma_s$ be a fixed Cauchy surface. Let 
\begin{equation}
\label{operator U(t,s) lorentzian}
U(t,s)\colon C^\infty_0(\Sigma_s)\to C^\infty_0(\Sigma_t)
\end{equation}
be the linear operator uniquely defined (modulo an infinitely smoothing operator) by the following properties.
\begin{enumerate}[(1)]
		\item[] \textbf{Property 1}.
		The Schwartz kernel $u(t,x,y;s)$ of $U(t,s)$ satisfies
		\begin{align}
		\label{property 1 U(t,s)}
			\Box^{(t,x)} u(t,x,y;s)=0\,,\qquad
			u(s,x,y;s)=\delta_{\Sigma_s}(x,y)\,, 
		\end{align}
		where $\delta_{\Sigma_s}(x,y)$ denotes the integral kernel of the identity operator on $C^\infty(\Sigma_s)$.
		\item[] \textbf{Property 2}.
		The Schwartz kernel $u(t,x,y;s)$, seen as a distribution in $\mathcal{D}^\prime(M\times M)$ satisfies
		\begin{align}
		\label{WF U}
			\operatorname{WF}(u)=\{(X,Y,k_X,-k_Y)\in T^*(M\times M)\setminus\{0\}\;|\; (X,k_X)\sim(Y,k_Y)\;\textrm{and}\;k_X\triangleright 0\}
		\end{align}
		where $\operatorname{WF}$ denotes the wavefront set, $\sim$ means that the point $X$ and $Y$ are connected by a lightlike geodesic $\gamma$ whose tangent vector at $X$ is $k_X$, and $k_Y$ is the parallel transport of $k_X$ from $X$ to $Y$ along $\gamma$. The symbol $\triangleright$ means that $k_X$ is future directed. 
	\end{enumerate}

The operator $U(t,s)$, effectively, maps initial data $C^\infty_0(\Sigma_s)$ on the `initial' Cauchy surface to spacelike compact wave solutions in $C^\infty(M)$. The technology developed in earlier sections, allows us to construct $U(t,s)$ explicitly and invariantly as a global oscillatory integral, up to an infinitely smoothing operator.

\begin{theorem}\label{Thm: wave propagator for globally hyperbolic spacetimes}
	Let $(M,g)\simeq(\mathbb{R}\times\Sigma,\beta dt^2-h_t)$ be a globally hyperbolic spacetime with compact Cauchy surface.
	Then the Schwartz kernel $u(t,x,y;s)$ of the operator $U(t,s)$ can be written, modulo an infinitely smooth function in all variables, as a single oscillatory integral
	\begin{equation}
		u(t,x,y;s)\overset{\mod C^\infty}{=}
		\dfrac{1}{(2\pi)^{d-1}}\int_{T^*_y\Sigma} e^{i \,\varphi(t,x;s,y,\eta)}\,\mathfrak{a}(t;s,y,\eta)\, \chi(t,x;s,y,\eta) \,w(t,x;s,y,\eta)\,d \eta,
	\end{equation}
	global in space $x$ and in time $t$, where
	\begin{enumerate}[(a)]
		\item $\varphi$ is the Lorentzian Levi-Civita phase function \eqref{Eqn: Lorentzian phase function};
		\item $\mathfrak{a}$ is a polyhomogeneous symbol of order zero depending smoothly on $t$,
		\[
		\mathfrak{a} \sim \sum_{j=0}^\infty \mathfrak{a}_{-j}, \qquad
		\mathfrak{a}_{-j}(t;s,y,\lambda\,\eta)=\lambda^{-j}\, \mathfrak{a}_{-j}(t;s,y,\eta)\quad \forall\lambda>0;
		\]
		\item $\chi$ is an infinitely smooth function on $M\times\mathbb{R}\times T'\Sigma$ satisfying
		\begin{enumerate}[(i)]		
			\item $\chi(t,x;s,y,\eta)=0$ on $\{(t,x;s,y,\eta) \,|\, \|\eta\|_{\Sigma_s}\leq 1/2\}$;
			\item $\chi(t,x;s,y,\eta)=1$ on the intersection between $\{(t,x;s,y,\eta) \,|\, \|\eta\|_{\Sigma_s} \geq 1\}$ and any conical neighbourhood of $\{(t,x^\ast(t;s,y,\eta),s,y,\eta)\}$ where $(t,x^*(t;s,y,\eta))=X^*(t;s,y,\eta)$ is the flow defined in equation \eqref{Eqn: Lorentzian flow} -- \textit{cf.} Lemma \ref{Lem: reparametrization of Lorentzian flow}.
			\item $\chi(t,x;s,y,\alpha\, \eta)=\chi(t,x;s,y,\eta)$ for $\alpha\geq 1$ on $\{ (t,x;s,y,\eta) \, | \, \|\eta\|_{\Sigma_s}\geq 1  \}$;
		\end{enumerate}
		
		\item the weight $w$ is defined as
		\begin{equation}
		w(t,x;s,y,\eta):=
		[\rho_{h_t}(x)]^{-1/2}\,[\rho_{h_s}(y)]^{-1/2}
		\left[
		{\det}^2\!
		\left(
		\varphi_{x^\alpha \eta_\beta}(t,x;s,y,\eta)
		\right)
		\right]^{1/4},
		\end{equation}
		where $\rho_h$ is the Riemannian density on the Cauchy surface (see \eqref{density}) and the branch of the complex root is chosen so that 
		\[
		\left.
		\arg \left[
		{\det}^2\!
		\left(
		\varphi_{x^\alpha \eta_\beta}(t,x;s,y,\eta)
		\right)
		\right]^{1/4} \right|_{t=s} =0\,.
		\]
	\end{enumerate}
\end{theorem}

\begin{proof}
The proof is a straightforward adaptation of  results from \cite{CLV}, combined with the properties of the Lorentzian Levi-Civita phase function.
\end{proof}

Let us stress that, thanks to the adoption of the Levi-Civita phase function, (a straightforward adaptation of) the algorithm presented in Subsection~\ref{the algorithm} uniquely determines the scalar, global symbol $\mathfrak{a}$ of the operator $U(t,s)$. Furthermore, all homogeneous components $\mathfrak{a}_{-j}$, $j\ge 0$, of $\mathfrak{a}$ are also guaranteed to be scalar.


\section{An application: Hadamard states}
\label{Sec: An application: Hadamard states}

In Section~\ref{Sec: Extension to general globally hyperbolic spacetimes} we have seen how the results of \cite{CLV} can be suitably generalised to an arbitrary globally hyperbolic spacetime $(M,g)$ with compact Cauchy surfaces and of arbitrary dimension $d=\dim M\geq 3$. In this section we shall discuss an application of our results by showing that starting from the wave propagator constructed in Theorem~\ref{Thm: wave propagator for globally hyperbolic spacetimes} one can obtain a representation of the integral kernel of a Hadamard state up to smooth errors. 

Hadamard states play a distinguished role in the algebraic formulation of quantum field theory, see {\it e.g.} \cite{Book_AQFT}, and \cite{Khavkine:2014mta} for a recent review. Given a globally hyperbolic spacetime $(M,g)$ and a normally hyperbolic second order linear partial differential operator $P:C^\infty(M)\to C^\infty(M)$, a Hadamard two-point function is fully specified by a bi-distribution $\omega_2\in D^\prime(M\times M)$ such that its integral kernel $\omega_2(X,X^\prime)$ satisfies
\begin{equation}\label{Eqn:2-pt-conditions}
\left\{
\begin{array}{ll}
(P\otimes\mathbb{I})\,\omega_2(X,X^\prime)=(\mathbb{I}\otimes P)\,\omega_2(X,X^\prime)&\in C^\infty(M\times M),\\ \omega_2(X,X^\prime)-\omega_2(X^\prime,X)-i \,G(X,X^\prime)&\in C^\infty(M\times M),
\end{array}\right.
\end{equation}
and
\begin{equation}\label{Eqn:WF}
	\operatorname{WF}(\omega_2)=\{(X,Y,k_X,-k_Y)\in T^*(M\times M)\setminus\{0\}\;|\; (X,k_X)\sim(Y,k_y)\;\textrm{and}\;k_X\triangleright 0\}
\end{equation}
where $G=G^+-G^-$ is the retarded-minus-advanced fundamental solution introduced in Lemma \ref{Lem: solution space isomorphism}. The notation in \eqref{Eqn:WF} was introduced after \eqref{WF U}. In a quantum field theory Hadamard bi-distributions, subject to a further positivity requirement, play a distinguished role in that they identify a natural class of physically meaningful Gaussian quantum states.

The characterisation of Hadamard bidistributions via \eqref{Eqn:2-pt-conditions} and \eqref{Eqn:WF} admits the following alternative and equivalent description proven by Radzikowski in \cite{Radzikowski:1996ei,Radzikowski:1996pa}. Consider $\omega_2\in\mathcal{D}^\prime(M\times M)$ satisfying equation \eqref{Eqn:2-pt-conditions} and let $\omega_2(X,X^\prime)$ indicate the associated integral kernel. The bi-distribution satisfies \eqref{Eqn:WF} if, for every geodesically convex open subset $\mathcal{O}\subseteq M$ and for every $X,X^\prime\in\mathcal{O}$, 
\begin{equation}\label{Eqn: 2-pt-integral-kernel}
\omega_2(X,X^\prime)=\lim_{\epsilon\to 0^+}\left(H_\epsilon(X,X^\prime)+W(X,X^\prime)\right),
\end{equation}
where the limit is understood in the weak sense. In the above expression $W$ is an infinitely smooth function in  $C^\infty(\mathcal{O}\times\mathcal{O})$ and 
\begin{gather}\label{Eq: Hadamard parametrix}
H_\epsilon(X,X^\prime):=\beta_n\dfrac{U_d(X,X^\prime)}{\sigma^{\frac{d-2}{2}}_\epsilon(X,X^\prime)} +\gamma_d V_d(X,X^\prime)\log\left(\frac{\sigma_\epsilon(X,X^\prime)}{\ell^2}\right)\,,\qquad
\end{gather}
where 
\begin{itemize}
	\item $\ell\in\mathbb{R}$ is a reference length to make the logarithm dimensionless.
	\item $\sigma_\epsilon(X,X^\prime)=\sigma(X,X^\prime)+ i\epsilon(t(X)-t(X^\prime))+\epsilon^2$. Here $\sigma$ is the geodesic distance between $X$ and $X^\prime$ while $t:\mathbb{R}\times\Sigma\to\mathbb{R}$ is the global temporal function introduced in Theorem \ref{Th:glob_hyp}.
	\item $\beta_d$ and $\gamma_d$ are normalisation constants such that
	\begin{eqnarray}
	\beta_d= -\frac{\Gamma\left(\frac{d}{2}\right)}{2\pi^{\frac{d}{2}}},& \gamma_d=(-1)^{\frac{d}{2}}\frac{2^{1-d}}{\pi^{\frac{d}{2}}\Gamma(\frac{d}{2})},&\textrm{d even}\\
	\beta_d= (-1)^{\frac{d+1}{2}}\frac{\pi^{\frac{2-d}{2}}}{2\Gamma\left(\frac{4-d}{2}\right)},& \gamma_d=0,&\textrm{d odd}.		
	\end{eqnarray}
	\item $U_d(X,X^\prime), V_d(X,X^\prime)\in C^\infty(\mathcal{O}\times\mathcal{O})$ admit an asymptotic expansion as power series of the geodesic distance 
	\begin{equation}\label{Eq: Hadamard coefficients}
	U_d=\sum\limits_{k=0}^\infty U_{d,k}\sigma^k,\quad V_d=\sum\limits_{k=0}^\infty V_{d,k}\sigma^k.
	\end{equation}
	All coefficients $U_{d,k}$ and $V_{d,k}$ are smooth and they are determined from the equation of motion \eqref{Eqn:2-pt-conditions} via a hierarchy of transport equations with prescribed initial conditions, see \cite{Radzikowski:1996ei,Radzikowski:1996pa}. 
\end{itemize}

With a slight abuse of notation one refers to $H\equiv H_\epsilon$ as the Hadamard parametrix. From a rigorous standpoint, $H$ identifies a distribution $\widehat{H}\in\mathcal{D}^\prime(\mathcal{O}\times\mathcal{O})$ via
$$\widehat{H}(f,f^\prime):=\lim_{\epsilon\to 0^+} H_{\epsilon}(f,f^\prime),\quad\forall f,f^\prime\in C^\infty_0(\mathcal{O}).$$
The above discussion can be summarised in the following theorem (see \cite{Radzikowski:1996ei,Radzikowski:1996pa} for the proof).
\begin{theorem}\label{Thm: Radzikowski theorem}
	Let $(M,g)$ be an $d$-dimensional globally hyperbolic spacetime and let $P$ be a normally hyperbolic second order linear partial differential operator. For any $\omega_2\in\mathcal{D}^\prime(M\times M)$ satisfying \eqref{Eqn:2-pt-conditions}, the following two conditions are equivalent:
	\begin{enumerate}
		\item the singular structure of $\omega_2$ is codified by \eqref{Eqn:WF},
		\item in every geodesically convex open neighbourhood $\mathcal{O}\subseteq M$ the integral kernel of $\omega_2$ is of the form \eqref{Eqn: 2-pt-integral-kernel}.
	\end{enumerate}
	If one of these two conditions is met, $\omega_2$ is said to be of \emph{Hadamard form}.
\end{theorem}

It is worth observing that, given a bi-distribution $\omega_2\in\mathcal{D}^\prime(M\times M)$, a direct use of \eqref{Eqn: 2-pt-integral-kernel} to verify the Hadamard condition is often impractical, since it involves the explicit control of the integral kernel of $\omega_2$ in every geodesic neighbourhood. This is why \eqref{Eqn: 2-pt-integral-kernel} plays a key role in many applications of Hadamard states, {\it cf.} \cite{Khavkine:2014mta}, \cite{Drago-Gerard-17}. 

In the following we shall prove that there exists an explicit connection between our representation of the wave propagator, Hadamard states and the associated Hadamard parametrix, \eqref{Eq: Hadamard parametrix}. In this analysis we shall work under a few additional assumptions, namely 
\begin{itemize}
	\item $(M,g)$ is a $d$-dimensional ultrastatic globally hyperbolic spacetime with compact Cauchy surface,
	\item the normally hyperbolic second order linear partial differential operator $P$ ruling the dynamics is of the form $\partial^2_t-\Delta$ where $\Delta$ is the Laplace--Beltrami operator acting on scalars associated with the metric $h$, {\it cf.} equation \eqref{metric general form BS}.  
\end{itemize}

This choice is made only for computational and presentation simplicity. \emph{Mutatis mutandis}, our next results hold true replacing $-\Delta$ with $E=-\Delta + V$, where $V$ is a time-independent smooth potential.

Under the above assumptions, we have the following.

\begin{theorem}\label{Thm: construction of Hadamard parametrix}
Consider the operator
	\begin{align}\label{Eq: Hadamard FIO}
	\Omega(\tau,s):=\frac12(\sqrt{-\Delta})^{-1/2}\,e^{-i\sqrt{-\Delta}(\tau-s)}\,,\quad s,\tau\in\mathbb{R},
	\end{align}	
	where $(\sqrt{-\Delta})^{-1/2}$ is the square root of the pseudoinverse of $-\Delta$, cf.\ \cite[Chapter 2, Section 2]{Rellich}.
Then the integral kernel $\omega_2$ of \eqref{Eq: Hadamard FIO} is of Hadamard form as a distribution in $\mathcal{D}'(M\times M)$.
\end{theorem}

\begin{proof}
	To start with, observe that $\sqrt{-\Delta}$ is an elliptic pseudodifferential operator of order $-1$. Hence  $\Omega(\tau,s)$ is a Fourier integral operator whose amplitude can be determined (non-invariantly) by \cite[Thm. 18.2]{shubin}.
	
	It is easy to see that $\omega_2$ satisfies 
	$$(P\otimes\mathbb{I})\omega_2=(\mathbb{I}\otimes P)\omega_2\in C^\infty(M\times M).$$
Put 
\[
\Lambda(f_1,f_2):=\omega_2(f_1,f_2)-\omega_2(f_1,f_2), \qquad \forall f_1,f_2\in \mathcal{D}(M).
\]
Of course, $\Lambda$ defines a distribution in $\mathcal{D}^\prime(M\times M)$. Denoting a generic point of $M\simeq\mathbb{R}\times\Sigma$ by $X=(t,x)$ as above, the integral kernel $\Lambda(X,X^\prime)$ satisfies
	\begin{gather}\label{Eq: Cauchy problem for causal propagator}
	(P\otimes\mathbb{I})\Lambda(X,X^\prime)=(\mathbb{I}\otimes P)\Lambda(X,X^\prime)\in C^\infty(M\times M)\\
	\Lambda(X,X^\prime)|_{t=t^\prime}\in C^\infty(\Sigma\times\Sigma)\,,\qquad
	\partial_t\Lambda(X,X^\prime)|_{t=t^\prime}-\delta_\Sigma(x,x^\prime)\in C^\infty(\Sigma\times\Sigma),
	\end{gather}
	where $\delta_\Sigma(x,x^\prime)$ is the integral kernel of the identity operator acting on scalar functions on the Cauchy surface $\{t\}\times\Sigma$. Hence $-i\Lambda$ must coincide up to smooth terms with the retarded-minus-advanced fundamental solution $G$, {\it cf.} Lemma \ref{Lem: solution space isomorphism}.
	
	Finally we can determine the wavefront set of $\omega_2$ since, being $(\sqrt{-\Delta})^{-1/2}$ an elliptic operator, $\operatorname{WF}(\omega_2)=\operatorname{WF}(u)$ where $u$ is the integral kernel of $U(\tau-s)$.
Now, on account of \cite{CLV}, we know that $\operatorname{WF}(u)$ is of the form \eqref{Eqn:WF}. Thus, by Theorem \ref{Thm: Radzikowski theorem} we can draw the sought-after conclusion.
\end{proof}

\begin{remark}
Theorem~\ref{Thm: construction of Hadamard parametrix} can in principle be recovered indirectly from results available in the literature --- see, for example, \cite[Sections~5.1 and~9.1]{Strohmaier:2018bcw} and \cite[Section~6.2]{baer}.
\end{remark}

Theorem \ref{Thm: construction of Hadamard parametrix} ensures that the Fourier Integral Operator $\Omega$ in equation \eqref{Eq: Hadamard FIO} identifies a distribution $\omega_2$ of Hadamard form. Hence, on account of Theorem \ref{Thm: Radzikowski theorem}, we can conclude that in every geodesically convex open neighbourhood $\mathcal{O}\subseteq M$ the integral kernel of $\omega_2$ differs from the Hadamard one, {\it cf.} \eqref{Eq: Hadamard parametrix}, by a smooth remainder.

In view of the results of Section \ref{Sec: The wave propagator in static spacetimes}, we can write
\begin{gather}
\Omega(t,0)
\overset{\mod \Psi^{-\infty}}{=}\int_\Sigma \int_{T^*_y\Sigma} \frac{e^{i \varphi(t,x;y,\eta)}}{(2\pi)^{n-1}}\,\mathfrak{f}(t;y,\eta)\,\chi(t,x;y,\eta) w(t,x;y,\eta) \,\left(\,\cdot\,\right)\,\rho_h(y)\,dy \,d\eta.\label{half wave divided by sqrt}
\end{gather}
An explicit formula for the scalar symbol $\mathfrak{f}$ can be obtained following the same procedure outlined in Section \ref{Sec: The wave propagator in static spacetimes}, albeit with a different initial condition.
In what follows we show that there exists a concrete relation between the coefficients in the homogeneous expansion of the amplitude $\mathfrak{f}$ from \eqref{half wave divided by sqrt} and those appearing in \eqref{Eq: Hadamard coefficients} for the expansion of the Hadamard parametrix \eqref{Eq: Hadamard parametrix}.

Consider the integral kernel
\begin{equation}
\label{integral kernel calculation hadamard}
\mathcal{I}_{\mathfrak{f}}(t,x;y):=\frac{1}{(2\pi)^{n-1}} \int_{T^*_y\Sigma} e^{i \varphi(t,x;y,\eta)}\,\mathfrak{f}(t;y,\eta)\,\chi(t,x;y,\eta)\,w(t,x;y,\eta)\,d\eta,
\end{equation}
with $X= (t,x)\in M$ and $y\in \Sigma$. The task at hand is to perform a local expansion of this integral kernel in a geodesic neighbourhood of $(0,y)$ in $M$. In what follows, we identify all Cauchy surfaces $\{t\}\times\Sigma$, $t\in\mathbb{R}$, with $\Sigma$ and we choose geodesic normal coordinates centred at $y$. Furthermore, we adopt the same coordinates for $x$ and $y$.

Since we are only interested in obtaining a local expansion for $x$ close to $y$ and small $t$, when carrying out the algorithm to compute $\mathfrak{f}$ we can choose $\varphi$ to be
\begin{equation}
\label{real Levi-Civita normal coord}
\varphi(t,x;y,\eta)=\langle \exp_y^{-1}(x),\eta\rangle-\|\eta\| t
\end{equation}
in normal coordinates centred at $y$, where $\|\,\cdot\,\|$ is the Euclidean norm, $\langle\,\cdot\,,\,\cdot\rangle$ the Euclidean pairing, and $\exp_y$ the exponential map at $y$. Substituting \eqref{real Levi-Civita normal coord} into \eqref{definition weight} we obtain
\begin{equation}
\label{weight normal coordiantes}
w(t,x;y,\eta)=[\rho_h(x)]^{-1/2}.
\end{equation}

\noindent As a consequence, the integral kernel \eqref{integral kernel calculation hadamard} turns into
\begin{equation}
\mathcal{I}_{\mathfrak{f}}(t,x;y)=\frac{1}{(2\pi)^{d-1} [\rho_h(x)]^\frac12} \int_{T^*_y\Sigma} e^{i \langle \exp_y^{-1}(x),\eta\rangle-i\|\eta\| t}\,\mathfrak{f}(t;y,\eta)\,\chi(t,x;y,\eta)\,d\eta
\end{equation}
where $x$ and $t$ are now such that $(x,t)$ lies in a convex geodesic neighbourhood of $(0,y)$. Furthermore, since $(M,g)$ is ultrastatic, we can assume without loss of generality that $t$ is positive. In addition we can drop the cut-off $\chi$ since it is equal to $1$ in a neighbourhood of the region where the phase is stationary. Dropping it does not affect the final result.

Let $\zeta:[0,+\infty)\to\mathbb{R}$ be a smooth cut-off such that $\zeta(r)=0$ for $r\le 1/2$ and $\zeta(r)=1$ for $r\ge 1$. Then, for every $N\in \mathbb{N}$ we have
\begin{equation}
\label{8 July 2019 formula 5}
\mathcal{I}_{\mathfrak{f}}\stackrel{\mod S^{-N-2}_{\mathrm{ph}}}{=}\frac{1}{(2\pi)^{d-1} [\rho_h(x)]^\frac12} \int_{T^*_y\Sigma} e^{i \langle \exp_y^{-1}(x),\eta\rangle-i\|\eta\| t}\,\sum_{k=0}^N\mathfrak{f}_{-1-k}(t;y,\eta)\,\zeta(\|\eta\|)\,d\eta\,.
\end{equation}
Switching to polar coordinates $(r, \omega)\in \mathbb{R}\times \mathbb{S}^{d-2}$ in $T^*_y\Sigma$ we obtain
\begin{gather}
\mathcal{I}_{\mathfrak{f}}=_N\frac{1}{(2\pi)^{d-1} [\rho_h(x)]^\frac12} \int_0^\infty \int_{\mathbb{S}^{d-2}} e^{i\,r(\langle \exp^{-1}_y(x),\omega\rangle -t)}\,\sum_{k=0}^N\frac{\chi(r)}{r^{1+k}}(t;y,\omega)\,r^{d-2}\,dr \,\operatorname{dVol}_{\mathbb{S}^{d-2}}\notag\\
=_N\frac{1}{(2\pi)^{d-1} [\rho_h(x)]^\frac12}\sum_{k=0}^N \int_0^\infty  e^{-i r t}\, r^{d-k-3}\,\chi(r) \int_{\mathbb{S}^{d-2}} e^{i\,r\langle \exp_y(x),\omega\rangle}\,\mathfrak{f}_{-1-k}(t;y,\omega) \,dr\,\operatorname{dVol}_{\mathbb{S}^{d-2}},\label{8 July 2019 formula 6}
\end{gather}
where $=_N$ is a shortcut notation for $\stackrel{\mod S^{-N-2}_{\mathrm{ph}}}{=}$. We observe that the function $f(\omega):=\langle \exp_y(x),\omega\rangle$ is stationary at $\omega_0=\pm \hat{x}$, where 
\[
\hat{x}:=[\operatorname{dist}_{\Sigma}(x,y)]^{-1} (\exp_y(x))^\flat.
\] 
Furthermore, it holds $\operatorname{Hess} f(\pm \hat{x})=\mp \operatorname{dist}_{\Sigma}(x,y)\,\operatorname{Id}$. This allows us to evaluate the spherical integral via the stationary phase formula for $r\to+\infty$ (see, for example, \cite[Theorem~C.1]{SaVa}) as
\begin{multline}
\label{8 July 2019 formula 8}
\int_{\mathbb{S}^{d-2}} e^{i\,r(\langle \exp_y(x),\omega\rangle)}\,\mathfrak{f}_{-1-k}(t;y,\omega) \,\operatorname{dVol}_{\mathbb{S}^{d-2}}\\
=\sum_{\kappa=\pm} e^{i\,\kappa\, r \dist_\Sigma(x,y)} e^{\frac{i \kappa(2-d)\pi}{4}} \left( \frac{2\pi}{r\,\dist_\Sigma(x,y)} \right)^{\frac{d-2}{2}}\sum_{j=0}^K r^{-j} [L_j \mathfrak{f}_{-1-k}](t;y,\kappa \hat{x}) +O(r^{-K-1})
\end{multline}
for every $K\in \mathbb{N}$, where
\begin{equation}
\label{8 July 2019 formula 9}
[L_j \mathfrak{f}_{-1-k}](t;y,\kappa \hat{x})=\left.\sum_{\substack{\mu,\nu\,:\,\nu-\mu=j\\3\mu\le2\nu}} \dfrac{2^{-\nu}}{\mu!\nu!}  \langle [\operatorname{Hess}f(\kappa \hat{x})]^{-1} D_\omega \,,\, D_\omega \rangle^\nu (f_{0,\kappa}^\mu \mathfrak{f}_{-1-k})(t;y,\omega)\right|_{\omega=\kappa \hat{x}},
\end{equation}
with
\begin{equation}
\label{8 July 2019 formula 10}
f_{0,\kappa}(\omega)=f(\omega)-f(\kappa \hat{x})-\frac12 \langle \operatorname{Hess}f(\kappa \hat{x})(\omega-\kappa\hat{x}), (\omega-\kappa\hat{x})\rangle.
\end{equation}
The local asymptotic expansion is now obtained by substituting \eqref{8 July 2019 formula 8}--\eqref{8 July 2019 formula 10} into \eqref{8 July 2019 formula 6} and subsequently carrying out the integral in $r$. We will compute the first few terms explicitly in the special case
\begin{equation}
\label{special case}
d=4, \quad M= \mathbb{R}\times \mathbb{S}^3.
\end{equation}
The argument in the general case would proceed along the same lines, only it would be a bit more lengthy and technically involute.

When $\Sigma=\mathbb{S}^3$, it is easy to see that the symbol $\mathfrak{f}$ in \eqref{half wave divided by sqrt} depends only on the magnitude of $\eta$, not on its direction. Hence, setting
\begin{equation}
\hat{\mathfrak{f}}(t;y):=\mathfrak{f}(t;y,\omega), \qquad \omega\in \mathbb{S}^2,
\end{equation}
it holds that, modulo $S_{\mathrm{ph}}^{-N-2}$,
\begin{equation}
\label{temp 1 1}
\begin{split}
\mathcal{I}_{\mathfrak{f}}
&
=_N
\frac{1}{(2\pi)^{3} [\rho_h(x)]^\frac12}\sum_{k=0}^N \int_0^{+\infty}  e^{-i r t}\, r^{1-k}\,\chi(r) \int_{\mathbb{S}^{2}} e^{i\,r\langle \exp_y(x),\omega\rangle}\,\hat{\mathfrak{f}}_{-1-k}(t;y) \,\operatorname{dVol}_{\mathbb{S}^{2}}\,dr\\
&=_N\frac{[\dist_{\Sigma}(x,y)]^{-1}}{4\pi^2 [\rho_h(x)]^\frac12}\sum_{k=0}^N \hat{\mathfrak{f}}_{-k-1}(t;y)\times\\
&\times \int_0^{+\infty}  \left( \sum_{j=0}^K \sum_{\kappa=\pm} \kappa\,\dfrac{e^{i r(\kappa\, \dist_\Sigma(x,y)- t)}}{i\,r^{k+j}}\,\chi(r) [L_j (1)](t;y,\kappa \hat{x}) +O(r^{-(k+K+1)}) \right) dr.
\end{split}
\end{equation}
By analysing  formula \eqref{8 July 2019 formula 10} for $f_{0,\kappa}$ and setting $\mathfrak{f}_{-1-k}=1$ in \eqref{8 July 2019 formula 9}, one obtains 
\begin{equation}
[L_{j}(1)](t;y,\kappa \hat{x})=\kappa^j \,[L_{j}(1)](t;y,\hat{x})
\end{equation}
so that \eqref{temp 1 1} becomes
\begin{equation}
\label{temp 1}
\begin{split}
\mathcal{I}_{\mathfrak{f}}
&
=
\frac{[\dist_\Sigma(x,y)]^{-1}}{4\pi^2 [\rho_h(x)]^\frac12}\sum_{k=0}^N \hat{\mathfrak{f}}_{-k-1}(t;y)\times\\
&\times \int_0^{+\infty}  \left( \sum_{j=0}^K \sum_{\kappa=\pm} \kappa^{1+j}\,\dfrac{e^{i r(\kappa\, \dist_\Sigma(x,y)- t)}}{i\,r^{k+j}}\,\chi(r) [L_j (1)](t;y, \hat{x}) +O(r^{-(k+K+1)}) \right) dr.
\end{split}
\end{equation}
The last step in the computation consists in evaluating, in a distributional sense, integrals of the form
\[
\int_{0}^{+\infty} \dfrac{e^{i r \widehat{\varphi}}}{r^n} \,dr, \qquad n\in \mathbb{N},
\]
where $\widehat{\varphi}=\kappa\, \dist_\Sigma(x,y)- t$.

For $n=0$ we have
\begin{equation}
\label{integral n=0}
\begin{split}
\int_0^{+\infty} e^{i r \widehat{\varphi}} \,dr&= \lim_{\varepsilon\to 0^+}\int_0^{+\infty} e^{i r (\widehat{\varphi}+i \varepsilon)} \,dr=\lim_{\varepsilon\to 0^+}\dfrac{i}{\widehat{\varphi}+i \varepsilon}.
\end{split}
\end{equation}

For $n\in \mathbb{N}\setminus\{0\}$ we have
\begin{equation}
\label{integral n>0}
\begin{split}
\int_0^{+\infty} \dfrac{e^{i r \widehat{\varphi}}}{r^n} \,dr&= \lim_{\varepsilon\to 0^+}\int_0^{+\infty} e^{i r (\widehat{\varphi}+i \varepsilon)} \dfrac{(-1)^{n+1}}{(n-1)!}\dfrac{d^n}{dr^n}\log r \,dr\\
&
=
\lim_{\varepsilon\to 0^+}\dfrac{- i^n \,(\widehat{\varphi}+i \varepsilon)^n}{(n-1)!}\int_0^{+\infty} e^{i r (\widehat{\varphi}+i \varepsilon)} \log r \,dr
\\
&
=
\lim_{\varepsilon\to 0^+}\dfrac{- i^{n-1}\,(\widehat{\varphi}+i \varepsilon)^{n-1}}{(n-1)!} \left[\gamma+\log(\varepsilon-i \widehat{\varphi})\right].
\end{split}
\end{equation}
When computing \eqref{integral n>0} we used the fact that, for $\Re z<0$,
\begin{equation}
\int_0^{+\infty} e^{z r} \log r\,dr=\dfrac{\gamma+\log(-z)}{z}.
\end{equation}

\noindent In view of \eqref{integral n=0}, the term corresponding to $j=k=0$ in \eqref{temp 1} is given by
\begin{equation}
\label{first term}
\begin{split}
\dfrac{1}{4\pi^2 [\rho_h(x)]^\frac12}\dfrac{\mathfrak{f}_{-1}(t;y,\hat{x})}{\dist_\Sigma(x,y)}\sum_{\kappa=\pm} \dfrac{\kappa}{\kappa \dist_\Sigma(x,y)-t}
&
=
\dfrac{1}{4\pi^2}\dfrac{\mathfrak{f}_{-1}(t;y,\hat{x})}{[\rho_h(x)]^\frac12} \dfrac{2}{\dist^2_\Sigma(x,y)-t^2}
\\
&
=
\dfrac{1}{8\pi^2}\dfrac{2\,\mathfrak{f}_{-1}(t;y,\hat{x})}{[\rho_h(x)]^\frac12} \dfrac{1}{\sigma((t,x),(0,y))}
\end{split}
\end{equation}

\noindent Since the principal symbol of the half-wave propagator is $1$ \cite{CLV} and since
\[
\left[(-\Delta)^{-\frac12}\right]_\mathrm{prin}(y,\eta)=\frac{1}{\sqrt{h^{\alpha\beta}(y),\eta_\alpha\eta_\beta}}
\]
\cite[Theorem~18.1]{shubin} implies that $\mathfrak{f}_{-1}(t;y,\hat{x})=\frac12$. It follows that \eqref{first term} corresponds to the first term in the expansion \eqref{Eq: Hadamard parametrix}
\begin{equation}
\label{first term bis}
\dfrac{1}{8\pi^2}\dfrac{[\rho_h(x)]^{-\frac12}}{\sigma((t,x),(0,y))}.
\end{equation}
Furthermore we recover the well-known fact that, given two generic points $X=(t,x),Y=(t^\prime,y)\in\mathbb{R}\times\Sigma$, the function $u$ appearing in \eqref{Eq: Hadamard parametrix} satisfies $u(t,y,t,y)=1$ ) and that $u(X,Y)$ coincides with the square root of the Van Vleck-Morette determinant, which, in normal coordinates centred at $y$, reads $[\rho_h(x)]^{-\frac12}$ \cite{Moretti99}.

\begin{remark}
	We would like to stress that, the calculation of the first term \eqref{first term bis} is general and works for an arbitrary four-dimensional ultrastatatic spacetime. In fact, we did not use anywhere the assumption that the Cauchy surface is a sphere. 
\end{remark}

\noindent The substitution of \eqref{first term bis} and \eqref{integral n>0} into \eqref{temp 1} yields
\begin{equation}
\label{temp 3}
\begin{split}
\mathcal{I}_\mathfrak{f}&\sim \dfrac{1}{8\pi^2}\dfrac{[\rho_h(x)]^{-\frac12}}{\sigma((t,x),(0,y))}\\
&
+
\frac{[\dist_\Sigma(x,y)]^{-1}}{4\pi^2 [\rho_h(x)]^\frac12}\sum_{k,j=0}^\infty \sum_{\kappa=\pm}\kappa^{1+j}\,\hat{\mathfrak{f}}_{-k-1}(t;y)[L_j (1)](t;y, \hat{x}) \times \\
&\times\lim_{\varepsilon\to 0^+}\dfrac{i^{j+k}\,(\widehat{\varphi}+i \varepsilon)^{j+k-1}}{(j+k-1)!}\left[\gamma+\log(\varepsilon-i \widehat{\varphi})\right].
\end{split}
\end{equation}
Let us write down explicitly the term with $j+k=1$. We have 
\begin{multline}
\label{next term 1}
\frac{[\dist_\Sigma(x,y)]^{-1}}{4\pi^2 [\rho_h(x)]^\frac12} \sum_{\kappa=\pm} i \left(\kappa \,\hat{\mathfrak{f}}_{-2}+\hat{\mathfrak{f}}_{-1}[L_{-1}(1)]  \right) \left[\gamma+\log(i(t-\kappa \dist_\Sigma(x,y)))\right]=\\
=i \frac{[\dist_\Sigma(x,y)]^{-1}}{4\pi^2 [\rho_h(x)]^\frac12} \left[ 
\hat{\mathfrak{f}}_{-2} \log\left(\dfrac{t-\dist_\Sigma(x,y)}{t+\dist_\Sigma(x,y) } \right)
+
\hat{\mathfrak{f}}_{-1}[L_{-1}(1)] \left(2\gamma+\log(\dist^2_\Sigma(x,y)-t^2)\right)
\right]\\
=\log(\dist_\Sigma^2(x,y)-t^2) \frac{i\,[\dist_\Sigma(x,y)]^{-1}}{4\pi^2 [\rho_h(x)]^\frac12}
\left( 
\hat{\mathfrak{f}}_{-1}[L_{-1}(1)] 
+\hat{\mathfrak{f}}_{-2}
\right)\\
+ \frac{i\,[\dist_\Sigma(x,y)]^{-1}}{4\pi^2 [\rho_h(x)]^\frac12}
\left[ 
\hat{\mathfrak{f}}_{-2}\left( i\pi+\log((t+\dist_\Sigma(x,y))^2) \right)+2\gamma\,\hat{\mathfrak{f}}_{-1}[L_{-1}(1)] 
\right].
\end{multline}
We observe that \eqref{next term 1} contributes a log term and a smooth term to the expansion. It is easy to see that subsequent contributions in \eqref{temp 3} yield terms of the form $v_k\, \sigma^k \log \sigma$ and smooth terms, where the functions $v_k$ are completely determined by the homogeneous components of the symbol $\mathfrak{f}$ and by the geometry. 

\section*{Acknowledgements}

M.C.~is grateful to Dima Vassiliev for enlightening conversations. The authors are grateful to Igor Khavkine and Alex Strohmaier for useful comments. The work of N.D.~is supported by a fellowship of the Alexander von Humboldt foundation and he is grateful to the University of Pavia and of Trento for the hospitality during the realisation of part of this work.

\appendix

\section{Fundamental Facts on Fundamental Solutions}
\label{App: Fundamental Facts on Fundamental Solutions}

In this appendix we recall, for completeness, a few basic well-known properties of the fundamental solutions of the operator $P$ defined in \eqref{Eq:KG}. 
In fact, these operators can be shown to exist for a larger class of backgrounds that those considered in this papers, 
for example when relaxing the assumption of compactness of the Cauchy surface or allowing for the presence of a timelike boundary, see, {\it e.g.}, \cite{Dappiaggi-Drago-Ferreira-19,Dappiaggi-Drago-Longhi-20,Grosse-Murro-18}.

\begin{theorem}\label{Thm:propagators}
	Let $(M,g)$ be a globally hyperbolic spacetime and let $P$ be the Klein--Gordon operator \eqref{Eq:KG}.
	Then $P$ possesses unique {\em advanced} $(-)$ and {\em retarded} $(+)$ {\em fundamental solutions} $G^\pm\in\mathcal{D}^\prime(M\times M)$ such that, calling $\mathcal{G}^\pm:\mathcal{D}(M)\to C^\infty(M)$ the associated maps via the Schwartz kernel theorem \cite[Thm. 8.2.12]{Hor},
	\begin{itemize}
		\item  $P\circ\mathcal{G}^\pm=\mathcal{G}^\pm\circ P=\operatorname{Id}|_{\mathcal{D}(M)}\,$,
		\item $\mathrm{supp}\,(\mathcal{G}^\pm(f))\subseteq J^\mp(\mathrm{supp}(f))$, for all $f\in \mathcal{D}(M)$, where $J^\pm(A)$ stand for the causal future $(+)$ and for the causal past $(-)$ of any open subset $A\subseteq M$.
	\end{itemize}
\end{theorem}

We refer the reader to \cite{Bar:2007zz} for the proof. The distribution $G:=G^+-G^-\in\mathcal{D}^\prime(M\times M)$ is called the \emph{retarded-minus-advanced fundamental solution}.

Using $G^\pm$, one can characterise the space of smooth solutions of \eqref{Eq:KG} as follows.

\begin{lemma}\label{Lem: solution space isomorphism}
	Let $(M,g)$ be a globally hyperbolic spacetime and let $\mathcal{G}^\pm:C^\infty_0(M)\to C^\infty(M)$ be the maps associated via the Schwartz kernel theorem to the advanced and retarded fundamental solutions $G^\pm$ of the Klein--Gordon operator $P$. Let  $\mathcal{S}_{sc}(M)=\{\Phi\in C^\infty_{sc}(M)\;|\;P\Phi=0\}$ where the subscript {\em sc} stands for spacelike compact\footnote{$\Phi\in C^\infty_{sc}(M)$ if it is smooth and the intersection between the support of $\Phi$ and any Cauchy surface is compact.}. Then there exists an isomorphism of topological vector spaces
	$$\mathcal{S}_{sc}(M)\simeq\frac{C^\infty_0(M)}{P[C^\infty_0(M)]},$$
	where the isomorphism is implemented by $\mathcal{G}\doteq \mathcal{G}^+-\mathcal{G}^-$ via the map $\frac{C^\infty_0(M)}{P[C^\infty_0(M)]}\ni [f]\mapsto \mathcal{G}(f)\in\mathcal{S}_{sc}(M)$.
\end{lemma}

\noindent In this paper, we are particularly interested in the map that associates to initial data of the Klein--Gordon equation $P$ the corresponding solution.

\begin{proposition}\label{Prop: initial data to solutions}
	Let $(M,g)$ be a globally hyperbolic spacetime and let $G$ be the retarded-minus-advanced fundamental solution associated the Klein--Gordon operator $P$. Let $\Sigma_{\overline{t}}\doteq\{\overline{t}\}\times\Sigma$ be an arbitrary but fixed Cauchy surface and let $n$ be its future pointing, normal vector of unit norm. Let $\iota_{\overline{t}}: \Sigma_{\overline{t}} \to M$ be the standard inclusion map. Let $\mathcal{G}_{0,\overline{t}}:C^\infty_0(\Sigma_{\overline{t}})\to C^\infty(M)$ and $\mathcal{G}_{1,\overline{t}}:C^\infty_0(\Sigma_{\overline{t}})\to C^\infty(M)$ be the linear maps obtained via the Schwartz kernel theorem respectively from $(\mathbb{I}\otimes\iota^*_{\overline{t},n}) G$ and $(\mathbb{I}\otimes\iota^*_{\overline{t}}) G$ where $\iota^*_{\overline{t}}$ is the pull-back map and $\iota^*_{\overline{t},n}:=\iota^*_{\overline{t}}\circ(n^\mu\nabla_\mu)$. Then, for every $f_0, f_1\in C^\infty_0(\Sigma_{\overline{t}})$, the Cauchy problem
	$$\left\{\begin{array}{l}
	P\Phi=0,\\
	\Phi|_{\Sigma_{\overline{t}}}=f_0,\quad (n^\mu\partial_\mu\Phi)|_{\Sigma_{\overline{t}}}=f_1,
	\end{array}
	\right.$$
	admits as unique solution
	\begin{equation}\label{Eqn: solution_as_convolution}
	\Phi= \mathcal{G}_{0,\overline{t}}(f_0)+\mathcal{G}_{1,\overline{t}}(f_1).
	\end{equation}
\end{proposition}

The proof of the above proposition is a direct adaptation to our notations and setting of \cite[Lemma 3.2.2]{Bar:2007zz}.


\begin{thebibliography}{23}
	
	\bibitem{Bar:2007zz}
	C.~B\"ar, N.~Ginoux and F.~Pf\"affle,
	{\it Wave equations on Lorenzian manifolds and quantization,}
	ESI Lectures in Mathematics and Physics, Eur. Math. Soc. (2007) 194p.	
	
	\bibitem{baer}
	C.~B\"ar and A.~Strohmaler,
	{\it ``An index theorem for Lorentzian manifolds with compact spacelike Cauchy boundary,''}
Amer.~J.~Math.  {\bf 141} no.~5 (2019) 1421--1455.
	
	\bibitem{Bernal:2004gm}
	A.~N.~Bernal and M.~Sanchez,
	{\it``Smoothness of time functions and the metric splitting of globally hyperbolic space-times,''}
	Comm.\ Math.\ Phys.\  {\bf 257} (2005) 43.
	
	\bibitem{Bernal:2005qf}
	A.~N.~Bernal and M.~Sanchez,
	{\it``Further results on the smoothability of Cauchy hypersurfaces and Cauchy time functions,''}
	Lett.\ Math.\ Phys.\  {\bf 77} (2006) 183.	
	
	\bibitem{Book_AQFT}
	R.~Brunetti, C.~Dappiaggi, K.~Fredenhagen and J.~Yngvason,
	{\it Advances in algebraic quantum field theory,} Springer (2015) 453p.
	
	\bibitem{CLV}
	M.~Capoferri, M.~Levitin and D.~Vassiliev,
	{\it``Geometric wave propagator on Riemannian manifolds,''}
	to appear in Comm. Anal. Geom., arXiv:1902.06982 [math.AP], 2019.
	
	\bibitem{CV}
	M.~Capoferri and D.~Vassiliev,
	{\it ``Global propagator for the massless Dirac operator and spectral asymptotics,''}
	arXiv:2004.06351 [math.AP], 2020.
	
	\bibitem{Dappiaggi-Drago-Ferreira-19}
	C.~Dappiaggi, N.~Drago and H.~Ferreira,
	{\it``Fundamental solutions for the wave operator on static Lorentzian manifolds with timelike boundary,''}
	Lett.~Math.~Phys. {\bf 109} no.~10 (2019)  2157.
	
	\bibitem{Dappiaggi-Drago-Longhi-20}
	C. Dappiaggi, N. Drago, R. Longhi,
	\textit{On Maxwell's Equations on Globally Hyperbolic Spacetimes with Timelike Boundary},
Ann. Henri Poincar\'e \textbf{21} (2020) 2367--2409.
	
	\bibitem{Drago-Gerard-17}
	N. Drago, C. G\'erard,
	\textit{On the adiabatic limit of Hadamard states},\\
	Lett.~Math.~Phys. {\bf 107} (2017) 1409.
	
	\bibitem{Fulling}
	S.~A.~Fulling,
	{\it``Aspects of Quantum Field Theory in Curved Space-time,''}
	London Math.\ Soc.\ Student Texts {\bf 17} (1989) 315p.
	
	\bibitem{Gerard-Wrochna-14}
	C.~Ge\'rard and M.~Wrochna,
	{\it ``Construction of Hadamard states by pseudo-differential calculus,''}
	Comm.~Math.~Phys.  {\bf 325} (2014).
	
	\bibitem{Grosse-Murro-18}
	N. Gro\ss e, S. Murro,
	\textit{``The well-posedness of the Cauchy problem for the Dirac operator on globally hyperbolic manifolds with timelike boundary},''
	arXiv:1806.06544 [math.DG], 2018.
	
	\bibitem{Junker:2001gx}
	W.~Junker and E.~Schrohe,
	{\it ``Adiabatic vacuum states on general space-time manifolds: Definition, construction, and physical properties",}
	Ann.~Henri Poincar\'e {\bf 3} (2002) 1113.
	
	\bibitem{Hor}
	L.~H\"ormander,
	{\it The analysis of linear partial differential operators. I}.
	Reprint of the second (1990) edition. Classics in Mathematics. Springer-Verlag, Berlin, 2003;
	{\it III}. Reprint of the 1994 edition. Classics in Mathematics. Springer-Verlag, Berlin, 2007; {\it IV}. Reprint of the 1994 edition. Classics in Mathematics. Springer-Verlag, Berlin, 2009.
	
	\bibitem{Husemoller-94}
	D. Husem\"oller,
	\textit{Fibre Bundles}, (1994) Springer, 354p.
	
	\bibitem{Kay:1978yp}
	B.~S.~Kay,
	{\it``Linear Spin 0 Quantum Fields in External Gravitational and Scalar Fields. 1. A One Particle Structure for the Stationary Case,''}
	Comm.\ Math.\ Phys.\  {\bf 62} (1978) 55.
	
	\bibitem{Khavkine:2014mta}
	I.~Khavkine and V.~Moretti,
	{\it``Algebraic QFT in Curved Spacetime and quasifree Hadamard states: an introduction,''}
	Chapter 5, Advances in Algebraic Quantum Field Theory, R. Brunetti
	et al. (eds.), Springer (2015).
	
	\bibitem{Poisson:2011nh}
	E.~Poisson, A.~Pound and I.~Vega,
	{\it ``The Motion of point particles in curved spacetime,''}
	Living Rev.~Rel.  {\bf 14} (2011) 7.
	
	\bibitem{LSV}
	A.~Laptev, Yu.~Safarov and D.~Vassiliev,
	{\it ``On global representation of Lagrangian distributions and solutions of hyperbolic equations,''}
	Comm.~Pure~Appl. Math. \textbf{47} 11 (1994) 1411.
	
	\bibitem{Moretti99}
	V.~Moretti, 
	{\it``Proof of the Symmetry of the Off-Diagonal Hadamard/Seeley-deWitt's Coefficients in $C^\infty$ Lorentzian Manifolds by a “Local Wick Rotation,''} Comm.\ Math.\ Phys.\  {\bf 212} 1 (2000) 165.
	
	\bibitem{Radzikowski:1996ei}
	M.~J.~Radzikowski,
	{\it ``A Local to global singularity theorem for quantum field theory on curved space-time,''}
	Comm.\ Math.\ Phys.\  {\bf 180} no.~1 (1996).
	
	\bibitem{Radzikowski:1996pa}
	M.~J.~Radzikowski,
	{\it ``Micro-local approach to the Hadamard condition in quantum field theory on curved space-time,''}
	Comm.\ Math.\ Phys.\  {\bf 179} (1996) 529.
	
	\bibitem{Rellich}
	F.~Rellich, {\it Perturbation theory of eigenvalue problems}, Courant Institute of Mathematical Sciences, New York University (1954) 127p.
		
	\bibitem{SaVa}
	Yu.~Safarov and D.~Vassiliev,
	{\it The asymptotic distribution of eigenvalues of partial differential operators}.
	Amer. Math. Soc. (1997) 370p.
	
	\bibitem{shubin}
	M.~A.~Shubin,
	{\it Pseudodifferential operators and spectral theory}, 2nd edition, Springer
	(2001) 288p.
	
	\bibitem{Strohmaier:2018bcw}
	A.~Strohmaier and S.~Zelditch,
	{\it ``A Gutzwiller trace formula for stationary space-times",}
	arXiv:1808.08425 [math.AP].
	
	\bibitem{Wald}
	R.~M.~Wald, {\it General Relativity}, Chicago University Press (1984) 491p.
	
	
	
\end{thebibliography}
\end{document}